\documentclass[11pt]{amsart}
\newfont{\cyr}{wncyr10 scaled 1100}
\usepackage{amsthm,amssymb,amsmath,amsfonts,mathrsfs,amscd, graphics}
\usepackage{mathtools}
\usepackage[latin1]{inputenc}

\usepackage[top=1.4in, bottom=1.4in, left=1.2in, right=1.2in]{geometry}
\usepackage[all]{xy}
\usepackage{latexsym}
\usepackage{longtable}
\usepackage{color}
\usepackage{tikz}
\usepackage{tikz-cd}
\usepackage{lua-visual-debug}
\usepackage{dsfont}
\usetikzlibrary{matrix}
\usepackage{hyperref}
\usepackage{enumitem}
\usepackage{setspace}

\newcommand*\ZZ{|[draw,circle]| \Z_2}
\numberwithin{equation}{section}
\usepackage[OT2,T1]{fontenc}
\DeclareSymbolFont{cyrletters}{OT2}{wncyr}{m}{n}
\DeclareMathSymbol{\Sha}{\mathalpha}{cyrletters}{"58}

\theoremstyle{plain}
\newtheorem{theorem}{Theorem}[section]
\newtheorem*{theorem*}{Theorem}
\newtheorem{corollary}[theorem]{Corollary}
\newtheorem{lemma}[theorem]{Lemma}
\newtheorem{proposition}[theorem]{Proposition}

\numberwithin{equation}{section}
\newtheorem{thm}{Theorem}
\newtheorem{ass}[thm]{Assumption}

\theoremstyle{definition}
\newtheorem{definition}[theorem]{Definition}

\newtheorem{remarkwr}[theorem]{Remark}

\theoremstyle{remark}
\newtheorem{obswr}[theorem]{Observation}

\newtheorem{intro-definition}[theorem]{Definition}

\newenvironment{remark}{\begin{remarkwr}\begin{upshape}}{\end{upshape}\end{remarkwr}}
\newenvironment{myproof}[2] {\paragraph{\emph{Proof of {#1} {#2} }}}{\hfill$\square$}

\def\Gal{\mathrm{Gal}}
\def\GL{\mathrm{GL}}

\def\ord{\mathrm{ord}}

\def\res{\mathrm{res}}

\def\sing{\mathrm{sing}}

\def\new{\mathrm{new}}
\def\ac{\mathrm{ac}}
\def\frakm{\mathfrak{m}}

\def\interX{\mathfrak{X}}

\def\ac{\mathrm{ac}}

\def\mix{\mathrm{mix}}
\def\un{\mathrm{un}}
\def\ram{\mathrm{ram}}
\def\con{\mathrm{cong}}

\def\FL{\mathrm{FL}}

\def\St{\mathrm{St}}
\def\prim{\mathrm{prim}}

\def\Ad{\mathrm{Ad}}
\def\leng{\mathrm{leng}}

\DeclareMathOperator{\Sym}{Sym}

\DeclareMathOperator{\Hom}{Hom}

\DeclareMathOperator{\End}{End}

\def\calD{\mathcal{D}}

\def\calH{\mathcal{H}}

\def\calM{\mathcal{M}}
\def\calP{\mathcal{P}}
\def\calO{\mathcal{O}}

\def\calS{\mathcal{S}}

\def\calZ{\mathcal{Z}}

\def\frakp{\mathfrak{p}}

\def\frakY{\mathfrak{Y}}

\def\Adel{\mathbf{A}}

\def\CC{\mathbf{C}}

\def\FF{\mathbf{F}}

\def\PP{\mathbf{P}}
\def\QQ{\mathbf{Q}}

\def\TT{\mathbb{T}}

\def\ZZ{\mathbf{Z}}

\def\rmA{\mathrm{A}}
\def\rmB{\mathrm{B}}
\def\rmC{\mathrm{C}}

\def\rmF{\mathrm{F}}
\def\rmG{\mathrm{G}}
\def\rmH{\mathrm{H}}
\def\rmI{\mathrm{I}}
\def\rmE{\mathrm{E}}
\def\rmK{\mathrm{K}}

\def\rmM{\mathrm{M}}
\def\rmN{\mathrm{N}}

\def\rmQ{\mathrm{Q}}
\def\rmU{\mathrm{U}}
\def\rmV{\mathrm{V}}
\def\rmR{\mathrm{R}}
\def\rmS{\mathrm{S}}
\def\rmT{\mathrm{T}}
\def\rmZ{\mathrm{Z}}
\def\rmX{\mathrm{X}}
\def\rmY{\mathrm{Y}}
\def\rmW{\mathrm{W}}

\def\sfT{\mathsf{T}}

\newcommand{\Iw}{\mathrm{Iw}}

\include{thebibliography}

\begin{document}

\title[Flach system and Jacquet--Langlands correspondence]
{Flach system of Jacquet--Langlands type and quaternionic period}

\author{Haining Wang }
\address{\parbox{\linewidth}{address:\\Shanghai Center for Mathematical Sciences,\\ Fudan University,\\No,2005 Songhu Road,\\Shanghai,200438, China.~ }}
\email{wanghaining1121@outlook.com}

\begin{abstract}
In this article, we study the adjoint Selmer group of a modular form that admits a Jacquet--Langlands transfer to a modular form on a definite quaternion algebra. We show that this Selmer group has length given by the valuation of the Petersson norm of the quaternionic modular form.  Our proof is based on an Euler system argument first used by Flach in his study for the adjoint Selmer group of an elliptic curve by producing elements in the motivic cohomology of the product of two Shimura curves. However our construction does not involve Siegel modular units which make it more adaptable to study adjoint motives coming from products of higher dimensional Shimura varieties.
\end{abstract}
   
\date{\today}

\maketitle
\tableofcontents
\section{Introduction}
In a seminal work of Flach \cite{Flach}, the author proves the finiteness of the adjoint Selmer group of an elliptic curve using an Euler system type argument following Kolyvagin. The key point is to construct suitable elements in the motivic cohomology of product of two modular curves with tight control of the local ramification behaviour when mapped to the Galois cohomology via certain Abel--Jacobi map. In fact, in \cite{Flach}, such element consists of a Hecke correspondence at a suitable auxiliary prime and a well-chosen Siegel modular unit. The local ramification behaviour of this element at the auxiliary prime reflects geometrically the Eichler--Shimura congruence relation. This work has been extended in various directions to the adjoint Selmer group of modular forms in \cite{Flach2} and \cite{Weston1} following similar strategies in Flach's original work. Recently, Flach's construction has been vastly extended to many automorphic motives and has been used to study Iwasawa theoretic question associated to these automorphic motives. We will refer to \cite{BDR1}, \cite{BDR2}, \cite{LZ}, \cite{LLZ}, \cite{LSZ} for some of these examples. 

In \cite{Flach}, the natural bound comes out of his Euler system argument is the degree of the modular parametrization for the elliptic curve. It is known also that the degree of modular parametrization is closely related to the congruence module of the modular form associated to the elliptic curve, see \cite{ARS} and \cite{Zagier}. In this article, we will revisit the construction of \cite{Flach} in the setting when the modular form admits Jacquet--Langlands transfer to a modular form on a definite quaternion algebra. In this case, we can incorporate techniques from arithmetic level raising to bound the Selmer group which manifest the Jacquet--Langlands correspondence at a geometric level. In particular, Siegel modular unit does not appear in our construction which makes it more adaptable to the setting of product of higher dimensional Shimura varieties and we refer to this Flach type system as a Flach system of Jacquet--Langlands type to distinguish it from the original Flach system which is of Eichler--Shimura type under the terminology of \cite{Weston1}. It turns out that in our setting the natural bound for the adjoint Selmer group is given by the Petersson norm of the quaternionic modular form obtained by a normalized Jacquet--Langlands transfer. We will also refer to this Petersson norm as the quaternionic period of the modular form and hence explains the title of this article. This bound is also reasonable as it is known this quaternionic period also measures congruences between modular forms, see in particular \cite{PW}, \cite{CH1} and \cite{KO} for such connections. 

\subsection{Main result}
In order to state our main results, we introduce some notations. These notations are sometimes simplified and may differ from what are used in the main body of this article. Throughout this article, we fix a prime $l\geq 5$. Let $f$ be a normalized newform in $\rmS^{\new}_{2}(\Gamma_{0}(\rmN))$. We further assume that we have a factorization $\rmN=\rmN^{+}\rmN^{-}$ such that $(\rmN^{+}, \rmN^{-})=1$ and $\rmN^{-}$ is square-free and consists of \emph{odd} number of prime factors. Let $\overline{\rmB}=\rmB_{\rmN^{-}}$ be the definite quaternion algebra of discriminant $\rmN^{-}$ and $\rmS^{\overline{\rmB}}(\rmN^{+})$ be the space of quaternionic modular forms for $\overline{\rmB}^{\times}$ of level $\rmN^{+}$.  Suppose $f^{\dagger}$ is the quaternionic modular form in $\rmS^{\overline{\rmB}}(\rmN^{+})$ corresponding to $f$ under the Jacquet--Langlands correspondence and is normalized integrally.  We denote by $\calP(f^{\dagger})=\langle f^{\dagger}, f^{\dagger}\rangle$ the Petersson norm of $f^{\dagger}$ which we will explain more below. Let $\rmE=\QQ(f)$ be the Hecke field of $f$ and we fix an embedding $\iota_{l}: \QQ^{\ac}\hookrightarrow \CC_{l}$ such that it induces a place $\lambda$ of $\rmE$. Let $\rmE_{\lambda}$ be the completion of $\rmE$ at $\lambda$ and $\calO=\calO_{\rmE_{\lambda}}$ be the valuation ring of $\rmE_{\lambda}$. We fix a uniformizer $\varpi\in\calO$ and write $\lambda=(\varpi)$ for its maximal ideal. We denote by $k_{\lambda}$ the residue field of $\calO$. We let $\TT=\TT_{\rmN^{+}, \rmN^{-}}$ be the $\lambda$-adic Hecke algebra corresponding to the cusp forms of level $\rmN=\rmN^{+}\rmN^{-}$ which is new at primes dividing $\rmN^{-}$. Since $f$ is an eigenform, we have a morphism
$\phi_{f}: \TT\rightarrow \calO$
corresponding to the system of Hecke eigenvalues of $f$. Let $\frakm$ be the maximal ideal given by the kernel $\mathrm{ker}\{\TT\xrightarrow{\phi_{f}} \calO\rightarrow \calO/\lambda\}$.  We will consider the Galois representation $\rho_{f,\lambda}$ attached to $f$ provided by the Eichler--Shimura construction
\begin{equation*}
\rho_{f,\lambda}: \rmG_{\QQ}\rightarrow \GL_{2}(\rmE_{\lambda})=\mathrm{Aut}(\rmV_{f,\lambda})
\end{equation*}
whose representation space is given by $\rmV_{\rho}=\rmV_{f,\lambda}$. We will sometimes abbreviate $\rho_{f,\lambda}$ by $\rho$ and we denote by $\overline{\rho}=\overline{\rho}_{f,\lambda}$ the residual representation of $\rho_{f,\lambda}$. We let $\Sigma^{+}$ be the set of primes dividing $\rmN^{+}$ and $\Sigma^{-}_{\ram}$ be the set of primes $v$ dividing $\rmN^{-}$ and such that $l\mid v^{2}-1$. Let  $\Sigma_{\mathrm{mix}}$ be the set of primes $v$ dividing $\rmN^{-}$ but $l\nmid v^{2}-1$. Recall for a deformation $\rho_{\rmA}$ of $\overline{\rho}$ over an $\calO$-algebra $\rmA$, we say $\rho_{\rmA}$ is minimally ramified at a prime $v$ if $\rho_{\rmA}(\rmI_{\QQ_{v}})$ has the same order as $\overline{\rho}(\rmI_{\QQ_{v}})$ when $l \nmid \# \overline{\rho}(\rmI_{\QQ_{v}})$ or if $\rho_{\rmA}$ has a rank one $\rmI_{\QQ_{v}}$-coinvariant over $\rmA$ when $l \mid \#\overline{\rho}(\rmI_{\QQ_{v}})$. We say $\bar{\rho}$ is minimal at the prime $v$ if all the liftings of $\bar{\rho}$ is minimally ramified at $v$.

\begin{ass}\label{ass1}
We make the following assumptions on $\bar{\rho}$. 
\begin{enumerate}
\item  $\bar{\rho}\vert_{\rmG_{\QQ(\zeta_{l})}}$ is absolutely irreducible;
\item The image of $\bar{\rho}$ contains $\GL_{2}(\FF_{l})$;
\item $\bar{\rho}$ is minimal at primes in $\Sigma^{+}$;
\item $\bar{\rho}$ is ramified at primes in $\Sigma^{-}_{\ram}$. 
\end{enumerate}
\end{ass}
Under this assumption, we will fix a Galois stable lattice $\rmT_{\rho}$ in $\rmV_{\rho}$ and define $\calM_{\rho}=\rmV_{\rho}/\rmT_{\rho}$. For each $n\geq1$, we put $\rmT_{n}=\rmT_{\rho}\mod \lambda^{n}$ and $\calM_{n}=\calM_{\rho}[\lambda^{n}]$. We  will be concerned with the adjoint representation $\Ad^{0}(\calM_{\rho})$ associated to $\calM_{\rho}$
which is also isomorphic to $\Sym^{2}(\calM_{\rho})(-1)$. We will work with the finite Galois module
\begin{equation*}
\rmM_{n}=\Ad^{0}(\calM_{n})=\Sym^{2}(\calM_{n})(-1) 
\end{equation*}
for the most of time. 

We will be concerned with the Bloch--Kato Selmer group 
\begin{equation*}
\rmH^{1}_{f}(\QQ, \Ad^{0}(\calM_{\rho}))=\rmH^{1}_{f}(\QQ, \mathrm{Ad}^{0}(\rho)\otimes \rmE_{\lambda}/\calO)
\end{equation*}
of  $\mathrm{Ad}^{0}(\calM_{\rho})$. Then our main result is the following theorem concerning the size of this group. 

\begin{theorem}
Suppose $f$ is a newform in $\rmS^{\new}_{2}(\Gamma_{0}(\rmN))$ which admits a normailized Jacquet--Langlands transfer to a quaternionic modular form $f^{\dagger}$ in $\rmS^{\overline{\rmB}}(\rmN^{+})$. Let $\eta=\varpi^{\nu}$ and $\nu=\mathrm{ord}_{\lambda}(\calP(f^{\dagger}))$ for the quaternionic period $\calP(f^{\dagger})$ of $f$. We assume that
\begin{enumerate}
\item the residual Galois representation $\overline{\rho}$ satisfies Assumption \ref{ass1};
\item the Galois cohomology group $\rmH^{1}(\QQ(\rmM_{n})/\QQ, \rmM_{n})=0$
for every $n\geq 1$ and where $\QQ(\rmM_{n})$ is the splitting field of the Galois module $\rmM_{n}$.
\end{enumerate}
Then the Selmer group $\rmH^{1}_{f}(\QQ, \Ad^{0}(\calM_{\rho}))$ is annihilated by $\eta$ and in fact has length equal to $\nu=\mathrm{ord}_{\lambda}(\calP(f^{\dagger}))$. 
\end{theorem}
We remark that the first assumption in the statement of the theorem is mainly used to prove results concerning arithmetic level raising theorems and certain $\rmR=\rmT$ theorem which we think can be weakened. The second assumption is inherited from the Euler system argument used in \cite{Flach} and \cite{Weston1} and this assumption is verified when the Galois representation $\rho$ is surjective onto its image. 

\subsection{Strategy of our proof} Next we describe the strategy for proving this theorem. To bound the Selmer group of $\Ad^{0}(\calM_{\rho})$, we will bound the Selmer group of $\rmM_{n}$ for each $n\geq 1$. We are naturally led to construct elements in the Galois cohomology group $\rmH^{1}(\QQ, \rmN_{n})$ where $\rmN_{n}=\Sym^{2}\rmT_{n}$ which is the Cartier dual of $\rmM_{n}$. 

In light of the arithmetic level raising theorems, we are led to consider the Shimura curve $\rmX(\rmB)$ where $\rmB=\rmB_{p\rmN^{-}}$ is the indefinite quaternion algebra of discriminant $p\rmN^{-}$ and where $p$ is the so-called $n$-admissible prime for $f$. These primes are level raising primes for $f$ such that the Galois representation $\rmT_{n}$ naturally appear in the first cohomology group $\rmH^{1}(\rmX(\rmB)\otimes\overline{\QQ}, \calO_{n}(1))$ of $\rmX(\rmB)$. Let $\TT^{[p]}=\TT_{\rmN^{+}, p\rmN^{-}}$ be the $\lambda$-adic Hecke algebra corresponding to the cusp forms of level $\rmN p=\rmN^{+}\rmN^{-}p$, which are new at primes dividing $\rmN^{-}p$. Since $p$ is $n$-admissible and hence level raising for $f$, the morphism $\phi_{f}$ modulo $\lambda$ gives rise to a maximal ideal $\frakm^{[p]}$ of $\TT^{[p]}$. Then $\rmH^{1}(\rmX(\rmB)\otimes\overline{\QQ}, \calO_{n}(1))_{\frakm^{[p]}}$ is in fact isomorphic to some copies of $\rmT_{n}$. Therefore $\rmN_{n}$ will appear naturally in 
\begin{equation*}
\rmH^{1}(\rmX(\rmB)\otimes\overline{\QQ}, \calO_{n}(1))^{\otimes 2}_{\frakm^{[p]}}
\end{equation*}
and we need to construct elements in 
\begin{equation*}
\rmH^{1}(\QQ, \rmH^{1}(\rmX(\rmB)\otimes\overline{\QQ}, \calO_{n}(1))^{\otimes 2}_{\frakm^{[p]}}). 
\end{equation*}
For this purpose, we consider the motivic cohomology group 
$\rmH^{3}_{\calM}(\rmX(\rmB)^{2},\ZZ(2))$
for the surface $\rmX(\rmB)^{2}$ whose elements consist of pairs $(\rmZ, f)$ where $\rmZ$ is a curve on $\rmX(\rmB)^{2}$ and $f$ is rational function on $\rmZ$ whose Weil divisor is trivial. 

Let $\theta: \rmX(\rmB)\rightarrow \rmX(\rmB)^{2}$ be the diagonal embedding, then we will define the \emph{Flach element} $\Theta^{[p]}(\rmB)$ to be the class in $\rmH^{3}_{\calM}(\rmX(\rmB)^{2},\ZZ(2))$ represented by the pair 
\begin{equation*}
\Theta^{[p]}(\rmB):=(\theta_{\ast}\rmX(\rmB), p).
\end{equation*}
 There is an Abel--Jacobi map 
\begin{equation*}
\mathrm{AJ}_{\frakm^{[p]}}: \rmH^{3}_{\calM}(\rmX(\rmB)^{2}, \ZZ(2))\rightarrow \rmH^{1}(\QQ, \rmH^{1}(\rmX(\rmB)\otimes\overline{\QQ},\calO(1))^{\otimes2}_{\frakm^{[p]}})
\end{equation*}
induced by the Chern character map from the coniveau spectral sequence in $K$-theory to that in \'etale cohomology.
We let 
\begin{equation*}
\kappa^{[p]}:=\mathrm{AJ}_{\frakm^{[p]}}(\Theta^{[p]}(\rmB))\in \rmH^{1}(\QQ, \rmH^{1}(\rmX(\rmB)\otimes\overline{\QQ},\calO(1))^{\otimes2}_{\frakm^{[p]}}) 
\end{equation*}
and we will refer to this element as the \emph{Flach class} for $\rmX(\rmB)^{2}$.
Composing the map $\mathrm{AJ}_{\frakm^{[p]}}$ with the natural projection map from $\rmH^{1}(\QQ, \rmH^{1}(\rmX(\rmB)\otimes\overline{\QQ},\calO(1))^{\otimes2}_{\frakm^{[p]}})$ to $\rmH^{1}(\QQ,\rmN_{n})$, we get our desired element 
$\kappa^{[p]}_{n}\in \rmH^{1}(\QQ, \rmN_{n})$. 

To apply the Euler system argument, we need to analyze the restriction of the element $\kappa^{[p]}_{n}$ in the local Galois cohomology $\rmH^{1}(\QQ_{v}, \rmN_{n})$ for each place $v$ of $\QQ$. Under our assumption, the most interesting place is $v=p$, and at this place we will calculate the depth of the singular residue $\partial_{p}(\kappa^{[p]}_{n})$ of  $\kappa^{[p]}_{n}$ in the singular part $\rmH^{1}_{\sin}(\QQ_{p}, \rmN_{n})$. Here the depth means the annihilator of the quotient of the rank one module $\rmH^{1}_{\sin}(\QQ_{p}, \rmN_{n})$ by the line  generated by $\partial_{p}(\kappa^{[p]}_{n})$. For this, we will prove the so-called reciprocity formula for $\kappa^{[p]}_{n}$. Let $f^{\dagger}$ be the normalized Jacquet--Langlands transfer of $f$ realized in the space $\Gamma(\rmZ(\overline{\rmB}), \calO)$ of $\calO$-valued functions on the Shimura set $\rmZ(\overline{\rmB})$. There is a natural pairing $\langle\cdot,\cdot\rangle$ on $\Gamma(\rmZ(\overline{\rmB}), \calO)_{\frakm}\times \Gamma(\rmZ(\overline{\rmB}), \calO)_{\frakm}$, we will call the period $\calP(f^{\dagger})=\langle f^{\dagger},f^{\dagger}\rangle$ the {\em Petersson norm} of $f^{\dagger}$ which we will also refer to it as the {\em quaternionic period of $f$}. 

\begin{theorem}
Suppose that $p$ is an $n$-admissible prime for $f$ and $\overline{\rho}$ satisfies Assumption \ref{ass1}. 
\begin{enumerate}
\item Then we have an isomorphism
\begin{equation*}
\rmH^{1}_{\mathrm{sin}}(\QQ_{p}, \rmH^{1}(\rmX(\rmB)\otimes\overline{\QQ}_{p},\calO_{n}(1))^{\otimes2}_{\frakm^{[p]}})
\cong \bigoplus\limits^{2}\limits_{i=1}\Gamma(\rmZ(\overline{\rmB}),\calO_{n})^{\otimes 2}_{\frakm}.
\end{equation*}

\item  Let  $\partial^{(i)}_{p}(\kappa^{[p]})$ be projection of $\partial_{p}(\kappa^{[p]})$ to the $i$-th copy in $\bigoplus\limits^{2}\limits_{i=1}\Gamma(\rmZ(\overline{\rmB}),\calO_{n})^{\otimes 2}_{\frakm}$ under the identification in $(1)$, then we have 
\begin{equation}\label{reci-intro}
(\partial^{(i)}_{p}(\kappa^{[p]}), f^{\dagger}\otimes f^{\dagger})=\calP(f^{\dagger}).
\end{equation} 
\end{enumerate}
\end{theorem}
Here the pairing on the left hand side of \ref{reci-intro} is the natural one on 
$\Gamma(\rmZ(\overline{\rmB}),\calO_{n})^{\otimes 2}_{\frakm}\times \Gamma(\rmZ(\overline{\rmB}),\calO_{n})^{\otimes 2}_{\frakm}$
induced by that of $\langle\cdot,\cdot\rangle$. This theorem is  enough to show that the depth of $\partial_{p}(\kappa^{[p]}_{n})$ in $\rmH^{1}_{\sin}(\QQ_{p}, \rmN_{n})$ is given by $\eta=\varpi^{\nu}$ with $\nu=\ord_{\lambda}(\langle f^{\dagger}, f^{\dagger}\rangle)$. For the other places, it can be shown that the $\res_{v}(\kappa^{[p]}_{n})$ lies in the spaces defining the local conditions defining the Bloch--Kato Selmer group. Then we can apply an Euler system argument as developed in \cite{Flach}, \cite{Weston1} to show the finite Bloch--Kato Selmer group $\rmH^{1}_{f}(\QQ, \rmM_{n})$ is annihilated by $\eta$ . Then we use this argument for each $n$ and we obtain that $\rmH^{1}_{f}(\QQ, \Ad^{0}(\calM_{\rho}))$ is annihilated by $\eta$.  To prove the main result, we are left to provide a lower bound for $\rmH^{1}_{f}(\QQ, \Ad^{0}(\calM_{\rho}))$, this is done by using the bound provided by the Taylor--Wiles method. More precisely, we introduce a smaller Selmer group $\rmH^{1}_{\calS}(\QQ, \Ad^{0}(\calM_{\rho}))$ following works of \cite{Lun}, \cite{KO} using an idea going back to Khare in \cite{Khare}. As a consequence of the Taylor--Wiles method one can prove a refined $\rmR=\rmT$ theorem which implies that this Selmer group has length equal to $\ord_{\lambda}(\eta(\rmN^{+},\rmN^{-}))$ for some congruence number $\eta(\rmN^{+},\rmN^{-})$ that detects congruences between $f$ and modular forms in $\rmS_{2}(\Gamma_{0}(\rmN))$ which are new at primes dividing $\rmN^{-}$. On the other hand, as a consequence of the freeness of $\Gamma(\rmZ(\overline{\rmB}),\calO)_{\frakm}$ over a suitable localized Hecke algebra. The congruence number $\eta(\rmN^{+},\rmN^{-})$  can be chosen to be $\eta$. Thus $\ord_{\lambda}(\eta)$ also provides a lower bound for $\leng_{\calO}\phantom{.}\rmH^{1}_{f}(\QQ, \Ad^{0}(\calM_{\rho}))$ and hence $\rmH^{1}_{f}(\QQ, \Ad^{0}(\calM_{\rho}))$ has length equal to  $\ord_{\lambda}(\eta)$.

From the above sketch, it follows that the elements $\{\kappa^{[p]}_{n}\}$ with $n$ varying and $p$ being an $n$-admissible prime  form a partial geometric Euler system of depth $\eta$ in the sense of Mazur--Weston \cite{Weston1}. It is also clear that arithmetic level raising and hence Jacquet--Langlands correspondence plays a prominent role in the construction and therefore we say the elements $\{\kappa^{[p]}_{n}\}$ form a Flach system of Jacquet--Langlands type. In a very analogues setting, we will show that such Flach system exists for certain quaternionic Hilbert modular surfaces \cite{Wang2}. In these two cases, the cycles are both given by Shimura curves which admit Drinfeld uniformization at an auxiliary prime and this makes the proof much easier as the rational function cuts out the entire special fiber which agree with the supersingular locus of the Shimura curve at this auxiliary prime. However we believe this is not a serious constraint and that Flach system of Jacquet--Langlands type should exist for adjoint motives appearing in the cohomology of product of Shimura varieties whenever the principle of Jacquet--Langlands transfer can be realized geometrically. For example, we will treat the case of product of unitary Shimura varieties studied in \cite{LTXZZ} and the case of product of quaternionic unitary Shimura varieties \cite{Wang3} studied by the author himself in a future work.

\subsection{Notations and conventions} We will use common notations and conventions in algebraic number theory and algebraic geometry. The cohomologies appeared in this article will be understood as the \'{e}tale cohomologies. For a field $\rmK$, we denote by $\overline{\rmK}$ the separable closure of $\rmK$ and put $\rmG_{\rmK}=\Gal(\overline{\rmK}/\rmK)$ the Galois group of $\rmK$. We let $\Adel$ be the ring of ad\`{e}les over $\QQ$ and $\Adel^{\infty}$ be the subring of finite ad\`{e}les.  For a prime $p$, $\Adel^{\infty, p}$ is the prime-to-$p$ part of  $\Adel^{\infty}$. 

When $\rmK$ is a local field, we denote by $\calO_{\rmK}$ its valuation ring  and by $k$ its residue field. We let $\rmI_{\rmK}$ be the inertia subgroup of $\rmG_{\rmK}$. For a $\rmG_{\rmK}$-module $\rmM$, 
\begin{enumerate}
\item the finite part $\rmH^{1}_{\mathrm{fin}}(\rmK, \rmM)$ of $\rmH^{1}(\rmK, \rmM)$ is defined to be $\rmH^{1}(k, \rmM^{\rmI_{\rmK}})$ 
\item the singular part $\rmH^{1}_{\mathrm{sin}}(\rmK, \rmM)$ of $\rmH^{1}(\rmK, \rmM)$ is defined to be the quotient of $\rmH^{1}(\rmK, \rmM)$ by the image of $\rmH^{1}_{\mathrm{fin}}(\rmK, \rmM)$ in $\rmH^{1}(\rmK, \rmM)$ via inflation.
\item Let $x$ be an element of $\rmH^{1}(\rmK, \rmM)$, we call the image of $x$ in $\rmH^{1}_{\mathrm{sin}}(\rmK, \rmM)$ the singular residue of $x$. 
\end{enumerate}
Let $\rmK$ be a number field and let $\rmK_{v}$ be the completion of $\rmK$ at a place $v$. Suppose $x\in\rmH^{1}(\rmK,\rmM)$, then we will write $\res_{v}(x)$ the image of $x$ under the restriction map 
\begin{equation*}
\res_{v}:\rmH^{1}(\rmK,\rmM)\rightarrow \rmH^{1}(\rmK_{v}, \rmM).
\end{equation*}
We will write $\partial_{v}(x)$ be the image of $x$ in $ \rmH^{1}_{\mathrm{sin}}(\rmK_{v}, \rmM)$ under the composite map of
\begin{equation*}
\res_{v}:\rmH^{1}(\rmK,\rmM)\rightarrow \rmH^{1}(\rmK_{v}, \rmM)
\end{equation*}
and the natural map $\rmH^{1}(\rmK_{v}, \rmM)\rightarrow \rmH^{1}_{\mathrm{sin}}(\rmK_{v}, \rmM)$.

\section{Monodromy filtration and Shimura curves}

\subsection{Shimura curves and Shimura sets}
Let $\rmN$ be a positive integer with a factorization $\rmN=\rmN^{+}\rmN^{-}$ with $\rmN^{+}$ and $\rmN^{-}$ coprime to each other. We assume that $\rmN^{-}$ is square-free and is a product of odd number of primes. Let $\overline{\rmB}$ be the definite quaternion algebra over $\QQ$ with discriminant $\rmN^{-}$ and let $\rmB$ be the indefinite quaternion algebra over $\QQ$ with discriminant $p\rmN^{-}$. Let $\calO_{\rmB}$ be a maximal order of $\rmB$ and $\calO_{\rmB, \rmN^{+}}$ be an Eichler order of level $\rmN^{+}$ in $\calO_{\rmB}$. We let $\rmG(\rmB)$ be the algebraic group over $\QQ$ given by $\rmB^{\times}$ and $\rmK_{\rmN^{+}}$ be the open compact of $\rmG(\rmB)(\mathbf{A}^{\infty})$ defined by $\widehat{\calO}^{\times}_{\rmB, \rmN^{+}}$. We let $\rmG(\overline{\rmB})$ be the algebraic group over $\QQ$ given by $\overline{\rmB}^{\times}$. Note that we have an isomorphism $\rmG(\rmB)(\mathbf{A}^{\infty, p})\xrightarrow{\sim} \rmG(\overline{\rmB})(\mathbf{A}^{\infty, p})$ and via this isomorphism we will view $\rmK^{p}$ as an open compact subgroup of $\rmG(\rmB)(\mathbf{A}^{\infty, p})$ for any open compact subgroup $\rmK$ of $\rmG(\overline{\rmB})(\mathbf{A}^{\infty})$. Let $d\geq 4$ be a prime such that $d\nmid p\rmN$. We define the open compact subgroup 
\begin{equation*}
\rmK=\rmK_{\rmN^{+}, d}= \{g=(g_{v})_{v}\in \widehat{\calO}^{\times}_{\rmB, \rmN^{+}}: g_{v}\equiv \begin{pmatrix} \ast & \ast\\ 0& 1\\\end{pmatrix}\mod v \text{ for } v=d\}.
\end{equation*}
Let $\rmX(\rmB)=\rmX_{\rmN^{+}, p\rmN^{-}, d}(\rmB)$ be the Shimura curve over $\QQ$ with level $\rmK=\rmK_{\rmN^{+}, d}$ explained above. The complex points of this curve is given by the following double coset
\begin{equation*}
\rmX(\rmB)(\CC)=\rmG(\rmB)(\QQ)\backslash \calH^{\pm} \times \rmG(\rmB)(\mathbf{A}^{\infty})/\rmK.
\end{equation*}
There is a natural model $\mathfrak{X}(\rmB)$ over $\ZZ[1/\rmN d]$ of $\rmX(\rmB)$ which coarsely represents the following moduli problem. Let $\rmS$ be a test scheme over $\ZZ[1/\rmN d]$ and then $\mathfrak{X}(\rmB)(\rmS)$ is the set of triples $(\rmA, \iota, \alpha_{\rmN^{+}d})$ up to isomorphism where
\begin{enumerate}
\item $\rmA$ is an abelian scheme over $\rmS$ of relative dimension $2$;
\item $\iota: \calO_{\rmB}\hookrightarrow \End(\rmA)$ is an $\calO_{\rmB}$-action on $\rmA$;
\item $\alpha_{\rmN^{+}d}$ is a $\rmK_{\rmN^{+}, d}$-level structure in the sense of \cite{Buzzard}: $\alpha_{\rmN^{+},d}$ is an orbit of the set of isomorphisms
\begin{equation*}
\alpha_{\rmN^{+}, d}: (\calO_{\rmB}\otimes\ZZ/\rmN^{+}d)_{\rmS}\cong \rmA[\rmN^{+}d]
\end{equation*} 
under the action of $\rmK=\rmK_{\rmN^{+}, d}$.
\end{enumerate}
One can also think of the $\rmK_{\rmN^{+}, d}$-level structure as the a pair $(\rmC_{\rmN^{+}}, \alpha_{d})$ consisting of the following datum:
\begin{enumerate}
\item $\rmC_{\rmN^{+}}$ is a finite flat subgroup scheme of $\rmA[\rmN^{+}]$ of order $(\rmN^{+})^{2}$ which is stable and locally cyclic under the action of $\calO_{\rmB}$;
\item $\alpha_{d}: (\ZZ/d\ZZ)^{\oplus 2}_{\rmS}\rightarrow \rmA[d]$ is an $\calO_{\rmB}$-equivariant injection of group schemes over $\rmS$. 
\end{enumerate}
It is well known this moduli problem is representable by a projective scheme over $\ZZ[1/p\rmN d]$ of relative dimension $1$. 

The  level $d$ structure in the above moduli problem plays a completely auxiliary role in this article. It rigidifies the above moduli problem and hence we get a fine moduli space.  However we are only interested in the space of modular forms of level $\rmK_{\rmN^{+}}$ defined by the Eichler order $\calO_{\rmB, \rmN^{+}}$ of $\calO_{\rmB}$, we will therefore also consider the Shimura curve $\rmX_{\rmN^{+}}(\rmB)$ whose $\CC$-points are given by the double coset space
\begin{equation*}
\rmX_{\rmN^{+}}(\rmB)(\CC)=\rmG(\rmB)(\QQ)\backslash \calH^{\pm} \times \rmG(\rmB)(\mathbf{A}^{\infty})/\rmK_{\rmN^{+}}.
\end{equation*}
This curve has a canonical model over $\QQ$. We will not rely on its integral model which is only defined by a coarse moduli space. There are two natural degeneracy maps 
\begin{equation*}
\pi_{0, d}: \rmX(\rmB)\rightarrow \rmX_{\rmN^{+}}(\rmB) \phantom{aa}\text{and}\phantom{aa}\pi_{1, d}: \rmX(\rmB)\rightarrow \rmX_{\rmN^{+}}(\rmB).
\end{equation*}
We will later choose the prime $d$ carefully so that there are no congruences between modular forms of level  $\rmK_{\rmN^{+}}$ and modular forms of level $\rmK_{\rmN^{+}, d}$ which are new at the prime $d$.

Let $\overline{\rmK}$ be the open compact subgroup of $\rmG(\overline{\rmB})(\mathbf{A}^{\infty})$ given by the Eichler order $\calO_{\overline{\rmB}, \rmN^{+}}$ and the level $d$ as above:
\begin{equation*}
\overline{\rmK}=\overline{\rmK}_{\rmN^{+},d}= \{g=(g_{v})_{v}\in \widehat{\calO}^{\times}_{\overline{\rmB}, \rmN^{+}}: g_{v}\equiv \begin{pmatrix} \ast & \ast\\ 0& 1\\\end{pmatrix}\mod v \text{ for } v=d\}.
\end{equation*}
We define the Shimura set $\rmZ(\overline{\rmB})$ by the following double coset
\begin{equation}\label{Shi-set}
\rmZ(\overline{\rmB})= \rmG(\overline{\rmB})(\QQ)\backslash \rmG(\overline{\rmB})(\mathbf{A}^{\infty})/\overline{\rmK}.
\end{equation}
Let $\overline{\rmK}_{0}(p)$ be the open compact subgroup of $\rmG(\overline{\rmB})(\mathbf{A}^{\infty})$ given by the Eichler order $\calO_{\overline{\rmB}, p\rmN^{+}}$ of level $p\rmN^{+}$ in $\overline{\rmB}$ and the level $d$ as above. We define $\rmZ_{\Iw(p)}(\overline{\rmB})$ to be the  double coset
\begin{equation}\label{Shi-set-p}
\rmZ_{\Iw(p)}(\overline{\rmB})= \rmG(\overline{\rmB})(\QQ)\backslash \rmG(\overline{\rmB})(\mathbf{A}^{\infty})/\overline{\rmK}_{0}(p).
\end{equation}
The role of $d$-level structure is also auxiliary as explained above in the curve case. To work with the genuine level of the modular form $f$, we also introduce the following Shimura sets. Let $\overline{\rmK}_{\rmN^{+}}$ be the open compact subgroup of $\rmG(\overline{\rmB})(\mathbf{A}^{\infty})$ given by the Eichler order $\calO_{\overline{\rmB}, \rmN^{+}}$, then we set
\begin{equation}\label{Shi-set}
\rmZ_{\rmN^{+}}(\overline{\rmB})= \rmG(\overline{\rmB})(\QQ)\backslash \rmG(\overline{\rmB})(\mathbf{A}^{\infty})/\rmK_{\rmN^{+}}.
\end{equation}
Note that there are two degeneracy maps 
\begin{equation*}
\pi_{0, d}: \rmZ(\overline{\rmB})\rightarrow \rmZ_{\rmN^{+}}(\overline{\rmB})\phantom{aa}\text{and}\phantom{aa}\pi_{1, d}: \rmZ(\overline{\rmB})\rightarrow \rmZ_{\rmN^{+}}(\overline{\rmB})
\end{equation*}
by forgetting the level structure at $d$.

Let $\Lambda$ be a ring. The cohomology group $\rmH^{0}(\rmZ(\overline{\rmB}), \Lambda)$ with coefficient in $\Lambda$ is simply the set of continuous function on the set $\rmZ(\overline{\rmB})$ valued in $\Lambda$. We will also write this space as $\Gamma(\rmZ(\overline{\rmB}),\Lambda)$ and refer to it as the space of  {quaternionic modular forms} on $\rmZ(\overline{\rmB})$ with values in $\Lambda$. The same notations apply to the Shimura sets $\rmZ_{\Iw(p)}(\overline{\rmB})$ and $\rmZ_{\rmN^{+}}(\overline{\rmB})$.

We define a bilinear pairing 
\begin{equation*}
\langle\cdot,\cdot\rangle: \Gamma(\rmZ(\overline{\rmB}),\calO)\times \Gamma(\rmZ(\overline{\rmB}),\rmE_{\lambda}/\calO)\rightarrow \calO
\end{equation*}
by the formula
\begin{equation*}
\langle \zeta, \phi\rangle= \sum_{z\in \rmZ(\overline{\rmB})} \vert\Gamma_{z}\vert^{-1}\zeta\phi(z)
\end{equation*}
for any $\zeta\in \Gamma(\rmZ(\overline{\rmB}), \calO)$ and $ \phi\in \Gamma(\rmZ(\overline{\rmB}), \rmE_{\lambda}/\calO)$. Here the finite group $\Gamma_{z}$ is defined by
\begin{equation*}
\overline{\rmB}^{\times}(\QQ)\cap z \rmK^{(p)}_{\rmN^{+}, d}(\calO_{\overline{\rmB}}\otimes\ZZ_{p})^{\times}z^{-1}
\end{equation*}
for any lift of $z$ in $\overline{\rmB}(\Adel^{\infty})$ denoted by the same symbol. Since $\rmK_{\rmN^{+},d}$ is neat, it follows that this factor is trivial and this is the other reason we would like to introduce this auxiliary integer $d$. Therefore this pairing simplifies to
\begin{equation*}
\langle \zeta, \phi\rangle= \sum_{z\in \rmZ(\overline{\rmB})} \zeta\phi(z).
\end{equation*}
The above paring induces an $\calO_{n}$-linear pairing
\begin{equation*}
 \langle\cdot, \cdot\rangle: \Gamma(\rmZ(\overline{\rmB}), \calO_{n})\times  \Gamma(\rmZ(\overline{\rmB}), \calO_{n})\rightarrow \calO_{n}.
\end{equation*}

We can also consider the above moduli problem over $\ZZ_{p^{2}}$  related to the Drinfeld uniformization of  $\rmX(\rmB)$.  We have to modify the second datum in the above moduli problem: we ask the embedding $\iota: \calO_{\rmB}\hookrightarrow \End(\rmA)$ to be special in the sense of \cite[131-132]{BC-unifor}.
We will denote the resulting projective scheme by the same notation $\mathfrak{X}(\rmB)$. Let $\overline{\rmX}(\rmB)$ be the special fiber of $\mathfrak{X}(\rmB)$ over $\FF_{p^{2}}$. The Cerednick--Drinfeld uniformization theorem asserts that the formal scheme ${\mathfrak{X}}(\rmB)^{\wedge}$ can be uniformized by the formal scheme $\breve{\calM}$ known as the Drinfeld upper-half plane:
\begin{equation}\label{p-unifor}
{\mathfrak{X}}(\rmB)^{\wedge} \xrightarrow{\sim} \rmG(\rmB)(\QQ)\backslash \breve{\calM} \times \rmG(\rmB)(\mathbf{A}^{\infty, p})/\rmK^{p}. 
\end{equation}
The above uniformization theorem implies that the model $\mathfrak{X}(\rmB)$ is regular with semi-stable reduction and all irreducible components of the special fiber $\overline{\rmX}(\rmB)$ are projective lines. More precisely, we have the following proposition describing the intersection behaviour of its irreducible components. In this article, we will use the following notations: let $\rmW$ be a scheme, we will then denote by  $\PP^{n}(\rmW)$ the $\PP^{n}$-bundle over $\rmW$. For example, $\PP^{1}(\rmZ(\overline{\rmB}))$ is the $\PP^{1}$-bundle over $\rmZ(\overline{\rmB})$.

\begin{proposition}\label{curve-red}
We have the following descriptions of the scheme $\overline{\rmX}(\rmB)$ in terms of the Shimura sets $\rmZ(\overline{\rmB})$ and $\rmZ_{\Iw(p)}(\overline{\rmB})$.
\begin{enumerate}
\item The scheme $\overline{\rmX}(\rmB)$ is a union of $\PP^{1}$-bundles over Shimura sets 
\begin{equation*}
\overline{\rmX}(\rmB)= \PP^{1}(\rmZ^{\circ}(\overline{\rmB}))\cup \PP^{1}(\rmZ^{\bullet}(\overline{\rmB})). 
\end{equation*}
where both $\rmZ^{\circ}(\overline{\rmB})$ and $\rmZ^{\bullet}(\overline{\rmB})$ are isomorphic to the Shimura set $\rmZ(\overline{\rmB})$. 
\item The two $\PP^{1}$-bundles $\PP^{1}(\rmZ^{\bullet}(\overline{\rmB}))$ and $\PP^{1}(\rmZ^{\circ}(\overline{\rmB}))$ intersect at the set of points which can be identified with the Shimura set
$\rmZ_{\Iw(p)}(\overline{\rmB})$.
This set of points also agree with the set of singular points on $\overline{\rmX}(\rmB)$. 
\end{enumerate}
\end{proposition}
\begin{proof}
This is a well-known result following from the Drinfeld uniformization. For example, see \cite[Proposition 3.2]{Wang}.
\end{proof}

We denote by 
\begin{equation*}
\pi^{\circ}: \rmZ_{\Iw(p)}(\overline{\rmB})\rightarrow \rmZ^{\circ}(\overline{\rmB})\phantom{aa}\text{and}\phantom{aa}\pi^{\bullet}: \rmZ_{\Iw(p)}(\overline{\rmB})\rightarrow \rmZ^{\bullet}(\overline{\rmB})
\end{equation*}
the natural specialization maps provided by $(2)$ of the above proposition. These two maps give the Hecke correspondence of $\rmZ(\overline{\rmB})$ at $p$. The nontrivial element in $\Gal(\FF_{p^{2}}/\FF_{p})$ acts on  $\overline{\rmX}(\rmB)$ and it permutes the two $\PP^{1}$-bundles $\PP^{1}(\rmZ^{\circ}(\overline{\rmB}))$ and $\PP^{1}(\rmZ^{\bullet}(\overline{\rmB}))$. On the set $\rmZ_{\Iw(p)}(\overline{\rmB})$ it acts by the classical {Atkin--Lehner involution}. For this, see \cite[\S 1.7]{BD-Mumford} for example. 

\subsection{Arithmetic level raising on Shimura curves} Let $f\in \rmS_{2}(\Gamma_{0}(\rmN))$ be a newform of weight $2$ with $\rmN$ as before. Let $\rmE=\QQ(f)$ be the Hecke field of $f$. Let $\lambda$ be a place of $\rmE$ above $l$ and $\rmE_{\lambda}$ be the completion of $\rmE$ at $\lambda$. Let $\varpi$ be a uniformizer of $\calO=\calO_{\rmE_{\lambda}}$ and we write $\calO_{n}=\calO/\varpi^{n}$ for $n\geq 1$. We let $\TT=\TT_{\rmN^{+}, \rmN^{-}}$ (respectively  $\TT^{[p]}=\TT_{\rmN^{+}, p\rmN^{-}}$) be the $\lambda$-adic Hecke algebra corresponding to the cusp forms of level $\rmN=\rmN^{+}\rmN^{-}$ (respectively of level $\rmN p=\rmN^{+}\rmN^{-}p$) which is new at primes dividing $\rmN^{-}$ (respectively at primes dividing $p\rmN^{-}$). This means $\TT$ contains the Hecke operators $\rmT_{v}$ for $v\nmid d\rmN$ and contains the Hecke operators $\rmU_{v}$ for $v\mid\rmN^{-}$. Similarly  $\TT^{[p]}$ contains the Hecke operators $\rmT_{v}$ for $v\nmid dp\rmN$ and contains the Hecke operators $\rmU_{v}$ for $v\mid p\rmN^{-}$. Since $f$ is an eigenform, we have a morphism
$\phi_{f}: \TT\rightarrow \calO$
corresponding to the system of Hecke eigenvalues $a_{v}(f)$ of $f$. More precisely, we have  $\phi_{f}(\rmT_{v})=a_{v}(f)$ for $v\nmid \rmN$  and  $\phi_{f}(\rmU_{v})=a_{v}(f)$ for $v\mid \rmN$.

\begin{definition}\label{n-adm}
Let $n\geq 1$ be an integer. We say that a prime $p$ is \emph{$n$-admissible} for $f$ if 
\begin{enumerate}
\item $p\nmid \rmN$ and $l\nmid p^{2}-1$;
\item $\varpi^{n}\mid p+1-\epsilon_{p}(f)a_{p}(f)$ for some $\epsilon_{p}(f)\in\{-1, 1\}$.
\end{enumerate}
\end{definition}

Let $\phi_{f, n}: \TT\rightarrow \calO_{n}$ be reduction of the map $\phi_{f}: \TT\rightarrow \calO$ modulo $\varpi^{n}$. We denote by $\mathfrak{p}_{f, n}$  the kernel of this map and $\mathfrak{m}_{f}$ the unique maximal ideal of $\TT$ containing $\frakp_{f, n}$. We will say $\mathfrak{m}_{f}$ is residually irreducible if the residue Galois representation $\bar{\rho}_{f,\lambda}$ attached to $f$  is absolutely irreducible. We will always assume that $\frakm_{f}$ is absolutely irreducible

\begin{definition}\label{clean}
 We say the auxiliary prime $d$ is \emph{clean} for $\overline{\rho}_{f,\lambda}$ if the degeneracy maps induce an isomorphism
\begin{equation*}
(\pi_{1,d, \ast}, \pi_{2,d, \ast}): \Gamma(\rmZ(\overline{\rmB}),\calO)_{/\frakm_{f}}\xrightarrow{\sim}  \Gamma(\rmZ_{\rmN^{+}}(\overline{\rmB}),\calO)^{\oplus 2}_{/\frakm_{f}}.
\end{equation*}
\end{definition}
Note that by Nakayma's lemma, if  $d$ is {clean} for $\overline{\rho}_{f,\lambda}$, then we have an isomorphism
\begin{equation}
(\pi_{1,d, \ast}, \pi_{2,d, \ast}): \Gamma(\rmZ(\overline{\rmB}),\calO)_{/\frakp_{f, n}}\xrightarrow{\sim}  \Gamma(\rmZ_{\rmN^{+}}(\overline{\rmB}),\calO)^{\oplus 2}_{/\frakp_{f, n}}.
\end{equation}

\begin{lemma}\label{clean-switch}
Suppose the prime $d$ is clean for $\overline{\rho}_{f,\lambda}$, the two degeneracy maps induce an isomorphism
\begin{equation*}
(\pi_{1, d, \ast}, \pi_{2, d,\ast}):  \rmH^{1}(\rmX(\rmB)\otimes\overline{\QQ}_{p}, \calO(1))_{/\frakp^{[p]}_{f,n}}\xrightarrow{\sim}  \rmH^{1}(\rmX_{\rmN^{+}}(\rmB)\otimes\overline{\QQ}_{p}, \calO(1))^{\oplus 2}_{/\frakp^{[p]}_{f,n}}
\end{equation*}
for each $n\geq1$
\end{lemma}
\begin{proof}
By Ihara's lemma \cite[Theorem 1]{DT} and the fact that $\rmX(\rmB)$ is \'etale over $\rmX_{\rmN^{+}}(\rmB)$, the map 
\begin{equation*}
(\pi_{1, d, \ast}, \pi_{2, d,\ast}):  \rmH^{1}(\rmX(\rmB)\otimes\overline{\QQ}_{p}, \calO(1))_{/\frakm^{[p]}}\xrightarrow{\sim}  \rmH^{1}(\rmX_{\rmN^{+}}(\rmB)\otimes\overline{\QQ}_{p}, \calO(1))^{\oplus 2}_{/\frakm^{[p]}}
\end{equation*}
is surjective. By the main result of \cite{BLR2}, we have the following
\begin{enumerate}
\item $\rmH^{1}(\rmX(\rmB)\otimes{\overline{\QQ}_{p}}, \calO(1))_{/\frakm^{[p]}_{f}}\cong \overline{\rho}^{\oplus d_{1}}_{f,\lambda}$ for some positive integer $d_{1}$; 
\item $\rmH^{1}(\rmX_{\rmN^{+}}(\rmB)\otimes{\overline{\QQ}_{p}}, \calO(1))_{/\frakm^{[p]}_{f}}\cong \overline{\rho}^{\oplus d^{\dagger}_{1}}_{f,\lambda}$ for some positive integer $d^{\dagger}_{1}$. 
\end{enumerate}
Suppose that $\dim_{k_{\lambda}}\Gamma(\rmZ(\overline{\rmB}),\calO)_{/\frakm_{f}}=d_{2}$ and $\dim_{k_{\lambda}}\Gamma(\rmZ_{\rmN^{+}}(\overline{\rmB}),\calO)_{/\frakm_{f}}=d^{\dagger}_{2}$. Then $d_{1}=d_{2}$ and $d^{\dagger}_{1}=d^{\dagger}_{2}$ by \cite[Lemma 6.4.2]{LTXZZ}. By the definition of $d$-cleaness in Definition \ref{clean}, $d_{2}=2d^{\dagger}_{2}$ and hence $d_{1}=2d^{\dagger}_{1}$. The result follows by Nakayama's lemma.
\end{proof}

The following result is known as the ramified arithmetic level raising theorem for Shimura curves and was first proved in \cite[Theorem 5.15, Corollary 5.18]{BD-Main} and see \cite[Theorem 3.4]{Wang} for a proof using Rapoport--Zink spectral sequence. 

\begin{theorem}\label{level-raise-curve}
Let $p$ be an $n$-admissible prime for $f$. We assume that the residual Galois representation $\bar{\rho}_{f,\lambda}$ satisfies Assumption \ref{ass1}.
Then there exists a surjective homomorphism $\phi^{[p]}_{f, n}: \TT^{[p]}\rightarrow \calO_{n}$ such that $\phi^{[p]}_{f, n}$ agrees with $\phi_{f, n}$ at all Hecke operators away from $p$ and sends $\rmU_{p}$ to $\epsilon_{p}(f)$. We will denote by $\frakp^{[p]}_{f, n}$ the kernel of $\phi^{[p]}_{f, n}$ and by $\frakm^{[p]}_{f}$ the unique maximal ideal containing $\frakp^{[p]}_{f, n}$. 
Moreover we have an isomorphism of $\calO_{n}$-modules 
\begin{equation*}
\Psi_{n}: \Gamma(\rmZ(\overline{\rmB}), \calO)_{/\frakp_{f, n}}\cong\rmH^{1}_{\sin}(\QQ_{p^{2}}, \rmH^{1}(\rmX(\rmB)\otimes{\overline{\QQ}_{p}}, \calO(1))_{/\frakp^{[p]}_{f, n}}).
\end{equation*}
\end{theorem}

\subsection{Monodromy filtration and deformation ring}
Let $\frakm^{?}=\frakm^{?}_{f}$ and $\frakp^{?}_{n}=\frakp^{?}_{f, n}$ for $?\in\{\emptyset, [p]\}$. We will set $\overline{\rho}=\overline{\rho}_{f,\lambda}$ and denote by $\overline{\rho}_{v}$ the restriction of $\overline{\rho}$ to $\rmG_{\QQ_{v}}$. It will follow from the freeness result of 
$\Gamma(\rmZ(\overline{\rmB}), \calO)_{\frakm}$ and $\rmH^{1}(\rmX(\rmB)\otimes{\overline{\QQ}_{p}}, \calO(1)))_{\frakm^{[p]}}$
over suitable deformation rings that we will discuss below that we can rewrite the above isomorphism as
\begin{equation*}
\Gamma(\rmZ(\overline{\rmB}), \calO_{n})_{\frakm}\xrightarrow{}\rmH^{1}_{\sin}(\QQ_{p^{2}}, \rmH^{1}(\rmX(\rmB)\otimes{\overline{\QQ}_{p}}, \calO_{n}(1))_{\frakm^{[p]}}).
\end{equation*}

Moreover using tools from Galois deformation theory, we can obtain a canonical presentation of 
$\rmH^{1}(\rmX(\rmB)\otimes{\overline{\QQ}_{p}}, \calO_{n}(1))_{\frakm^{[p]}}$
in terms of  $\Gamma(\rmZ(\overline{\rmB}), \calO_{n})_{\frakm}$. We will review this following \cite[\S 3.4]{Wang1} next. 

We consider the decomposition 
\begin{equation*}
\begin{aligned}
\Gamma(\rmZ_{\Iw(p)}(\overline{\rmB}), \calO)&=\Gamma(\rmZ(\overline{\rmB}), \mathrm{Ind}^{\GL_{2}(\ZZ_{p})}_{\Iw(p)}(\calO))\\
&=\Gamma(\rmZ(\overline{\rmB}), \St(\calO))\oplus\Gamma(\rmZ(\overline{\rmB}), \calO).\\
\end{aligned}
\end{equation*}
We define the {primitive part} $\Gamma(\rmZ_{\Iw(p)}(\overline{\rmB}), \calO_{\lambda})^{\prim}$ of $\Gamma(\rmZ_{\Iw(p)}(\overline{\rmB}), \calO_{\lambda})$ using the exact sequence 
\begin{equation}\label{first-prim-sub}
0\rightarrow \Gamma(\rmZ_{\Iw(p)}(\overline{\rmB}), \calO)^{\prim}\rightarrow \Gamma(\rmZ_{\Iw(p)}(\overline{\rmB}), \calO)\xrightarrow{(\pi^{\circ}_{\ast}\oplus \pi^{\bullet}_{\ast})}  \Gamma(\rmZ(\overline{\rmB}), \calO_{\lambda})\oplus \Gamma(\rmZ(\overline{\rmB}), \calO)\rightarrow 0
\end{equation}
Similarly, we define the {primitive quotient} $\Gamma(\rmZ_{\Iw(p)}(\overline{\rmB}), \calO)_{\prim}$ of $\Gamma(\rmZ_{\Iw(p)}(\overline{\rmB}), \calO)$ using the exact sequence 
\begin{equation}\label{second-prim-quo}
0\rightarrow \Gamma(\rmZ(\overline{\rmB}), \calO)\oplus \Gamma(\rmZ(\overline{\rmB}), \calO)\xrightarrow{(\pi^{\circ\ast}\oplus \pi^{\bullet\ast})} \Gamma(\rmZ_{\Iw(p)}(\overline{\rmB}), \calO)\rightarrow  \Gamma(\rmZ_{\Iw(p)}(\overline{\rmB}), \calO)_{\prim}\rightarrow 0.
\end{equation}
The decomposition of $\Gamma(\rmZ_{\Iw(p)}(\overline{\rmB}), \calO)=\Gamma(\rmZ(\overline{\rmB}), \St(\calO))\oplus\Gamma(\rmZ(\overline{\rmB}), \calO)$ allows us to rewrite the first exact sequence \eqref{first-prim-sub} as 
\begin{equation}\label{prim-1}
0\rightarrow \Gamma(\rmZ(\overline{\rmB}), \St(\calO))^{\prim}\rightarrow \Gamma(\rmZ(\overline{\rmB}), \St(\calO))\rightarrow  \Gamma(\rmZ(\overline{\rmB}), \calO)\rightarrow 0
\end{equation}
which allows us to identify the {primitive part} $\Gamma(\rmZ_{\Iw(p)}(\overline{\rmB}), \calO)^{\prim}$ with the primitive part $\Gamma(\rmZ(\overline{\rmB}), \St(\calO))^{\prim}$ of $\Gamma(\rmZ(\overline{\rmB}), \St(\calO))$. The decomposition also allows us to rewrite the second exact sequence \eqref{second-prim-quo} as
\begin{equation}\label{prim-2}
0\rightarrow \Gamma(\rmZ(\overline{\rmB}), \calO)\rightarrow \Gamma(\rmZ(\overline{\rmB}), \St(\calO))\rightarrow  \Gamma(\rmZ(\overline{\rmB}), \St(\calO))_{\prim} \rightarrow 0\\
\end{equation}
which identifies the {primitive quotient}  $\Gamma(\rmZ_{\Iw(p)}(\overline{\rmB}), \calO)_{\prim}$ with the primitive quotient $\Gamma(\rmZ(\overline{\rmB}), \St(\calO))_{\prim}$ of $\Gamma(\rmZ(\overline{\rmB}), \St(\calO))$. 

The monodromy filtration of $\rmH^{1}(\rmX(\rmB)\otimes\overline{\QQ}_{p}, \calO(1))_{\frakm^{[p]}}$ is calculated using the Rapoport--Zink spectral sequence whose first page is given by
\begin{equation}
\begin{tikzpicture}[thick,scale=0.8, every node/.style={scale=0.8}]
  \matrix (m) [matrix of math nodes,
    nodes in empty cells,nodes={minimum width=5ex,
    minimum height=5ex,outer sep=-5pt},
    column sep=1ex,row sep=1ex]{
                &      &     &     & \\
          2     &  \rmH^{0}(\rmZ_{\Iw(p)}(\overline{\rmB}), \calO(-1)) &  \rmH^{2}(\PP^{1}(\rmZ^{\circ}(\overline{\rmB})), \calO)\oplus\rmH^{2}(\PP^{1}(\rmZ^{\bullet}(\overline{\rmB})), \calO)  & & \\
          1     &       &  \rmH^{1}(\PP^{1}(\rmZ^{\circ}(\overline{\rmB})), \calO)\oplus\rmH^{1}(\PP^{1}(\rmZ^{\bullet}(\overline{\rmB})), \calO)&    & \\
          0     &    &\rmH^{0}(\PP^{1}(\rmZ^{\circ}(\overline{\rmB})), \calO)\oplus\rmH^{0}(\PP^{1}(\rmZ^{\bullet}(\overline{\rmB})), \calO)  &  \rmH^{0}(\rmZ_{\Iw(p)}(\overline{\rmB}), \calO) &\\
    \quad\strut &   -1  &  0  &  1  & \strut \\};
\draw[thick] (m-1-1.east) -- (m-5-1.east) ;
\draw[thick] (m-5-1.north) -- (m-5-5.north) ;
\end{tikzpicture}
\end{equation}
where we have omitted the localization at the maximal ideal $\frakm^{[p]}$. Therefore the monodromy filtration
\begin{equation}\label{mono2}
\begin{aligned}
&0\subset^{\rmE^{1,0}_{2, \frakm^{[p]}}} \rmM_{1}\rmH^{1}(\rmX(\rmB)\otimes{\overline{\QQ}_{p}}, \calO(1))_{\frakm^{[p]}} &\subset^{\rmE^{0,1}_{2, \frakm^{[p]}}} \rmM_{0}\rmH^{1}(\rmX(\rmB)\otimes{\overline{\QQ}_{p}}, \calO(1))_{\frakm^{[p]}}\\
&\subset^{\rmE^{-1,2}_{2, \frakm^{[p]}}} \rmM_{-1}\rmH^{1}(\rmX(\rmB)\otimes{\overline{\QQ}_{p}},\calO(1))_{\frakm^{[p]}}\\
\end{aligned}
\end{equation}
is given by
\begin{equation*}
\begin{aligned}
&\rmE^{1,0}_{2,\frakm^{[p]}}=\Gamma(\rmZ(\overline{\rmB}), \St(\calO(1)))_{\prim}^{\frakm^{[p]}};\\
&\rmE^{0,1}_{2, \frakm^{[p]}}= 0;\\
&\rmE^{-1,2}_{2, \frakm^{[p]}}= \Gamma(\rmZ(\overline{\rmB}), \St(\calO))^{\prim}_{\frakm^{[p]}}.\\
\end{aligned}
\end{equation*} 

For $?=\mathrm{ram}, \mathrm{un}, \mathrm{mix}$, we consider the global deformation problem given by
\begin{equation*}
\calS^{?}=(\overline{\rho}, \chi_{l}, \Sigma^{+}\cup \Sigma^{-}_{\ram}\cup\Sigma^{-}_{\mix}\cup\{p\}\cup\{l\}, \{\calD_{v}\}_{v\in \Sigma^{+}\cup \Sigma^{-}_{\ram}\cup\Sigma^{-}_{\mix}\cup\{p\}\cup\{l\}})
\end{equation*}
where 
\begin{enumerate}
\item for $v\in\Sigma^{+}\cup\Sigma^{-}_{\ram}$, $\calD_{v}$ is the local deformation problem classifying all liftings of  $\overline{\rho}_{v}$ that are minimally ramified;
\item for $v=p$, $\calD_{p}$ is the local deformation problem $\calD^{?}_{p}$ of $\overline{\rho}_{p}$  in the sense of \cite[Definition 3.51]{LTXZZa};
\item for $v\in \Sigma^{-}_{\mix}$, $\calD_{v}$ is the local deformation problem $\calD^{\mathrm{ram}}_{v}$ of $\overline{\rho}_{v}$  in the sense of \cite[Definition 3.51]{LTXZZa};
\item  for $v=l$, $\calD_{l}$ is the local deformation problem $\calD^{\FL}_{l}$ of $\overline{\rho}_{l}$ in the sense of \cite[Definition 3.51]{LTXZZa}.
\end{enumerate}
Then we obtain the global deformation rings ${\rmR}^{\un}$, ${\rmR}^{\ram}$ and ${\rmR}^{\mix}$ with relations given by
\begin{enumerate}
\item ${\rmR}^{\un}={\rmR}^{\mix}/(x)$;
\item ${\rmR}^{\ram}={\rmR}^{\mix}/(s-p^{2})$.
\end{enumerate}
We also define ${\rmR}^{\con}={\rmR}^{\ram}\otimes_{{\rmR}^{\mix}}{\rmR}^{\un}={\rmR}^{\mix}/(x, s-p^{2})$.  Let ${\sfT}^{\ram}$ be the image of $\TT^{[p]}$ in $\End_{\calO}(\rmH^{1}(\rmX(\rmB)\otimes\overline{\QQ}, \calO))$.
By Theorem \ref{level-raise-curve}, we have an isomorphism 
\begin{equation}\label{ram-level-map}
\Psi_{n}:\Gamma(\rmZ(\overline{\rmB}),\calO)_{/\frakp_{n}}\cong \rmH^{1}_{\sin}(\QQ_{p^{2}}, \rmH^{1}(\rmX(\rmB)\otimes\overline{\QQ}_{p}, \calO(1))_{/\frakp^{[p]}_{n}}).
\end{equation}
We know that ${\sfT}^{\ram}_{\frakm^{[p]}}\neq 0$ by this isomorphism.  Let ${\sfT}^{\un}$ be the image of $\TT$ in $\End_{\calO}(\Gamma(\rmZ(\overline{\rmB}),\calO))$. 
\begin{lemma}\label{free}
Suppose that Assumption \ref{ass1} holds and $d$ is clean for $\frakm$. We have the following statements.
\begin{enumerate}
\item We have an isomorphism ${\sfT}^{\un}_{\frakm}\cong{\rmR}^{\un}$ which makes $\Gamma(\rmZ(\overline{\rmB}),\calO)_{\frakm}$ a free ${\rmR}^{\un}$-module of rank two.

\item We have an isomorphism ${\sfT}^{\ram}_{\frakm^{[p]}}\cong{\rmR}^{\ram}$ which makes $\rmH^{1}(\rmX(\rmB)\otimes\overline{\QQ}, \calO)_{\frakm^{[p]}}$ a free ${\rmR}^{\ram}$-module of rank four.
\end{enumerate}
\end{lemma}
\begin{proof}
The same proof of \cite[Theorem 3.6.3]{LTXZZa} with obvious modifications shows that
\begin{enumerate}
\item $\Gamma(\rmZ_{\rmN^{+}}(\overline{\rmB}),\calO)_{\frakm}$ is a free ${\rmR}^{\un}=\sfT^{\un}_{\frakm}$-module;
\item $\rmH^{1}(\rmX_{\rmN^{+}}(\rmB)\otimes\overline{\QQ}, \calO)_{\frakm^{[p]}}$ is a free ${\rmR}^{\ram}={\sfT}^{\ram}_{\frakm^{[p]}}$-module.
\end{enumerate}
In fact, under the Assumption \ref{ass1}, $\Gamma(\rmZ_{\rmN^{+}}(\overline{\rmB}),\calO)_{\frakm}$ is a free $\sfT^{\un}_{\frakm}$ module of rank $1$ by \cite[Proposition 6.8]{CH1}. Therefore by $d$-cleanness, $\Gamma(\rmZ(\overline{\rmB}),\calO)_{\frakm}$ is free of rank $2$ over 
$\sfT^{\un}_{\frakm}$. By the same proof of Lemma \ref{clean-switch}, it follows that $\rmH^{1}(\rmX_{\rmN^{+}}(\rmB)\otimes\overline{\QQ}, \calO)_{\frakm^{[p]}}$ is a free ${\sfT}^{\ram}_{\frakm^{[p]}}$-module of rank $2$ and hence by Lemma \ref{clean-switch} $\rmH^{1}(\rmX(\rmB)\otimes\overline{\QQ}, \calO)_{\frakm^{[p]}}$ is a free ${\sfT}^{\ram}_{\frakm^{[p]}}$-module of rank $4$.
\end{proof}

It follows from the $n$-admissibility of $p$ that $\rmH^{1}(\rmX(\rmB)\otimes\overline{\QQ}_{p}, \calO_{n}(1))_{\frakm^{[p]}}$ is unramified as a $\rmG_{\QQ_{p}}$-module and thus we can regard it as a $\rmG_{\FF_{p^{2}}}$-module. 

\begin{proposition}\label{split}
There is a canonical split short exact sequence of $\rmG_{\FF_{p^{2}}}$-modules
\begin{equation*}
0\rightarrow\Gamma(\rmZ(\overline{\rmB}),\calO_{n}(1))_{\frakm}\rightarrow \rmH^{1}(\rmX(\rmB)\otimes\overline{\QQ}_{p}, \calO_{n}(1))_{\frakm^{[p]}}\rightarrow \Gamma(\rmZ(\overline{\rmB}),\calO_{n})_{\frakm}\rightarrow 0.
\end{equation*}
\end{proposition}
\begin{proof}

The result is first obtained in \cite[Proposition 3.11]{Wang1}, we will review the main ingredients here as some of these will be used later. First, the module $\Gamma(\rmZ(\overline{\rmB}), \mathrm{St}(\calO))_{\frakm^{[p]}}$ is free over ${\rmR}^{\mix}$ using the same method of the proof for \cite[Theorem 3.6.3]{LTXZZa}. It has the following properties
\begin{enumerate}
\item we have canonical isomorphisms
\begin{equation*}
\begin{aligned}
&\Gamma(\rmZ(\overline{\rmB}), \mathrm{St}(\calO))_{\frakm^{[p]}}\otimes_{{\rmR}^{\mix}}{\rmR}^{\un}\cong\Gamma(\rmZ(\overline{\rmB}),(\calO))_{\frakm}\\
&\Gamma(\rmZ(\overline{\rmB}), \mathrm{St}(\calO))_{\frakm^{[p]}}\otimes_{{\rmR}^{\mix}}{\rmR}^{\ram}\cong\Gamma(\rmZ(\overline{\rmB}),\St(\calO))^{\prim}_{\frakm^{[p]}}\\
&\Gamma(\rmZ(\overline{\rmB}), \mathrm{St}(\calO))_{\frakm^{[p]}}\otimes_{{\rmR}^{\mix}}{\rmR}^{\ram}\cong\Gamma(\rmZ(\overline{\rmB}),\St(\calO))_{\prim}^{\frakm^{[p]}};\\
\end{aligned}
\end{equation*}
\item there is a short exact sequence of ${\rmR}^{\ram}$-modules
\begin{equation*}
0\rightarrow\Gamma(\rmZ(\overline{\rmB}),\St(\calO_{\lambda})(1))^{\frakm^{[p]}}_{\prim}\rightarrow \rmH^{1}(\rmX(\rmB)\otimes\overline{\QQ}_{p}, \calO_{\lambda}(1))_{\frakm^{[p]}}\rightarrow \Gamma(\rmZ(\overline{\rmB}),\St(\calO_{\lambda}))_{\frakm^{[p]}}^{\prim}\rightarrow 0.
\end{equation*}
\end{enumerate}

The first property follows from tensoring two the exact sequences in \ref{prim-1} and \ref{prim-2} over ${\rmR}^{\mix}$ by ${\rmR}^{\un}$ and by ${\rmR}^{\ram}$ respectively.  The second part follows directly from the monodromy filtration \eqref{mono2} of $\rmH^{1}(\rmX(\rmB)\otimes\overline{\QQ}_{p}, \calO(1))_{\frakm^{[p]}}$. Here we are relying on the following facts which can be easily seen from the definitions and the Jacquet--Langlands correspondence:
\begin{enumerate}
\item the module $\Gamma(\rmZ(\overline{\rmB}),\St(\calO)(1))^{\frakm^{[p]}}_{\prim}$ and  $\Gamma(\rmZ(\overline{\rmB}),\St(\calO)(1))_{\frakm^{[p]}}^{\prim}$ are supported on the quotient ${\rmR}^{\ram}$ of ${\rmR}^{\mix}$ and they are naturally isomorphic;
\item the module $\Gamma(\rmZ(\overline{\rmB}), \calO(1))_{\frakm}$ is supported on the quotient ${\rmR}^{\un}$ of $\rmR^{\mix}$;
\item the module $\rmH^{1}(\rmX(\rmB)\otimes\overline{\QQ}_{p}, \calO(1))_{\frakm^{[p]}}$ is supported on the quotient $\rmR^{\ram}$ of ${\rmR}^{\mix}$.
\end{enumerate}

Next we tensor the exact sequence in $(2)$ above
by ${\rmR}^{\con}$ over ${\rmR}^{\ram}$. It follows from the freeness of $\rmH^{1}(\rmX(\rmB)\otimes\overline{\QQ}_{p}, \calO(1))_{\frakm^{[p]}}$ over $\rmR^{\ram}$ that $\Gamma(\rmZ(\overline{\rmB}),\St(\calO)(1))^{\frakm^{[p]}}_{\prim}$ and $\Gamma(\rmZ(\overline{\rmB}),\St(\calO))^{\prim}_{\frakm^{[p]}}$ are free over ${\rmR}^{\ram}$. Thus we have a short exact sequence 
\begin{equation*}
0\rightarrow\Gamma(\rmZ(\overline{\rmB}),\calO(1))_{\frakm}\otimes{\rmR}^{\con}\rightarrow \rmH^{1}(\rmX(\rmB)\otimes\overline{\QQ}_{p}, \calO(1))_{\frakm^{[p]}}\otimes{\rmR}^{\con}\rightarrow \Gamma(\rmZ(\overline{\rmB}),\calO)_{\frakm}\otimes{\rmR}^{\con}\rightarrow 0
\end{equation*}
where we have used $(1)$ to identify both
\begin{equation*}
\Gamma(\rmZ(\overline{\rmB}),\St(\calO))^{\frakm^{[p]}}_{\prim}\otimes_{{\rmR}^{\ram}}{\rmR}^{\con}\phantom{aa}\text{and}\phantom{aa}\Gamma(\rmZ(\overline{\rmB}),\St(\calO))_{\frakm^{[p]}}^{\prim}\otimes_{{\rmR}^{\ram}}{\rmR}^{\con}
\end{equation*}
with $\Gamma(\rmZ(\overline{\rmB}),\calO_{\lambda}))_{\frakm}\otimes_{{\rmR}^{\un}} {\rmR}^{\con}$. Let $\frakp^{\con}_{n}$ be the image of $\frakp_{n}$ in ${\rmR}^{\con}$ under the identification ${\rmT}^{\un}\cong {\rmR}^{\un}$. Note we have an isomorphism ${\rmR}^{\con}/\frakp^{\con}_{n}\cong \calO_{n}$. Using the freeness of $\Gamma(\rmZ(\overline{\rmB}),\calO)_{\frakm}$ over ${\rmR}^{\un}$ again, we conclude that we have the desired short exact sequence 
\begin{equation*}
0\rightarrow\Gamma(\rmZ(\overline{\rmB}),\calO_{n}(1))_{\frakm}\rightarrow \rmH^{1}(\rmX(\rmB)\otimes\overline{\QQ}_{p}, \calO_{n}(1))_{\frakm^{[p]}}\rightarrow \Gamma(\rmZ(\overline{\rmB}),\calO_{n})_{\frakm}\rightarrow 0.
\end{equation*}
by reducing the ideal  $\frakp^{\con}_{n}$.
\end{proof}

\begin{remarkwr} \label{sub-quotient}
The following remarks will be used in the following development.
\begin{enumerate}
\item Using this split exact sequence, we obtain
\begin{equation*}
\begin{aligned}
&\rmH^{1}_{\sing}(\QQ_{p^{2}}, \rmH^{1}(\rmX(\rmB)\otimes\overline{\QQ}_{p}, \calO_{n}(1))_{\frakm^{[p]}}) \cong (\rmH^{1}(\rmX(\rmB)\otimes\overline{\QQ}_{p},\calO_{n})_{\frakm^{[p]}})^{\rmG_{\FF_{p^{2}}}}\\
&\cong (\Gamma(\rmZ(\overline{\rmB}),\calO_{n}(1))_{\frakm}\oplus \Gamma(\rmZ(\overline{\rmB}),\calO_{n})_{\frakm})^{\rmG_{\FF_{p^{2}}}} \\
&\cong \Gamma(\rmZ(\overline{\rmB}),\calO_{n})_{\frakm}\\
\end{aligned}
\end{equation*}
which recovers the isomorphism 
\begin{equation*}
\Psi_{n}:\Gamma(\rmZ(\overline{\rmB}),\calO)_{/\frakp_{n}}\cong \rmH^{1}_{\sin}(\QQ_{p^{2}}, \rmH^{1}(\rmX(\rmB)\otimes\overline{\QQ}_{p}, \calO(1))_{/\frakp^{[p]}_{n}})
\end{equation*}
in Therorem \ref{level-raise-curve}.

\item In the presentation of the splitting exact sequence
\begin{equation}\label{sub-quo-split}
0\rightarrow\Gamma(\rmZ(\overline{\rmB}),\calO_{n}(1))_{\frakm}\rightarrow \rmH^{1}(\rmX(\rmB)\otimes\overline{\QQ}_{p}, \calO_{n}(1))_{\frakm^{[p]}}\rightarrow \Gamma(\rmZ(\overline{\rmB}),\calO_{n})_{\frakm}\rightarrow 0,
\end{equation}
the sub-part $\Gamma(\rmZ(\overline{\rmB}),\calO_{n}(1))_{\frakm}$ in \eqref{sub-quo-split}  comes from a quotient of 
\begin{equation*}
\bigoplus\limits_{?\in\{\circ, \bullet\}}\rmH^{0}(\PP^{1}(\rmZ^{?}(\overline{\rmB})), \calO_{n}(1))_{\frakm^{[p]}}
\end{equation*}
and the quotient-part $\Gamma(\rmZ(\overline{\rmB}),\calO_{n})_{\frakm}$ in \eqref{sub-quo-split} comes from a quotient of 
\begin{equation*}
\bigoplus\limits_{?\in\{\circ, \bullet\}}\rmH^{2}(\PP^{1}(\rmZ^{?}(\overline{\rmB})), \calO_{n}(1))_{\frakm^{[p]}}.
\end{equation*}
\end{enumerate}
\end{remarkwr}

\section{Construction of the Flach system}
\subsection{The motivic cohomology group} 
Let $\rmX$ be a proper smooth variety of finite type over a field $\rmF$. For an integer $d$, consider the complex
\begin{equation}\label{M-complex}
\bigoplus_{x\in \rmX^{d-1}}\rmK_{2}k(x)\rightarrow \bigoplus_{x\in \rmX^{d}}k(x)^{\times}\xrightarrow{d_{1}} \bigoplus_{x\in\rmX^{d+1}}\ZZ
\end{equation}
where $\rmX^{i}$ denotes the set of points of codimension $i$ on the variety $\rmX$, $\rmK_{2}k(x)$ is the Milnor $K$-group of the field $k(x)$, the first map is the so-called tame symbol map and the second is given by the divisor map. The \emph{motivic cohomology group} 
\begin{equation*}
\rmH^{2d+1}_{\calM}(\rmX, \ZZ(d+1))
\end{equation*}
is defined to be the cohomology of this complex: elements of $\rmH^{2d+1}_{\calM}(\rmX, \ZZ(d+1))$ are represented by formal sums $\sum_{i}(\rmZ_{i}, f_{i})$ of pairs of codimension $d$ cycles $\rmZ_{i}$ on $\rmX$ and non-zero rational functions $f_{i}$ such that $\sum_{i}\mathrm{div}_{\rmZ_{i}}(f_{i})=0$ as a Weil divisor on $\rmX$. 

This group is also known as the higher Chow group $\mathrm{CH}^{d}(\rmX, 1)$ of $\rmX$.  There is a Chern character map
\begin{equation}\label{chern}
\mathrm{ch}: \rmH^{2d+1}_{\calM}(\rmX, \ZZ(d+1))\rightarrow \rmH^{2d+1}(\rmX, \ZZ_{l}(d+1))
\end{equation}
given by the Chern character map from the coniveau spectral sequence in $K$-theory to that of \'etale cohomology. Next suppose that $\calO$ is a complete local ring with fraction field $\rmF$ and residue field $k$ of characteristic $p$ different from $l$. Let $\interX$ be a proper regular scheme over $\calO$, $\rmX$ be its generic fiber and $\overline{\rmX}$ be its special fiber over $k$. The motivic cohomology $\rmH^{2d}_{\calM}(\overline{\rmX}, \ZZ(d))$ in this case agree with the usual Chow group $\mathrm{CH}^{d}(\overline{\rmX})$. There is a map
\begin{equation}\label{div}
\mathrm{div}: \rmH^{2d+1}_{\calM}(\rmX, \ZZ(d+1))\rightarrow \rmH^{2d}_{\calM}(\overline{\rmX}, \ZZ(d))
\end{equation}
defined by sending a pair $(\rmZ, f)$ to the divisor of $f$ on the closure $\calZ$ of $\rmZ$ in $\interX$. Note that this divisor is purely supported on the special fiber $\overline{\rmZ}$ of $\calZ$.

\subsection{Regular model of the product of Shimura curves} 
Recall we have the integral model $\interX(\rmB)$ which is regular but non-smooth over $\ZZ_{p^{2}}$. We will consider the surface $\interX(\rmB)^{2}$ but this will be not be a regular scheme anymore. We can construct a regular model $\mathfrak{Y}(\rmB)$ of $\rmX(\rmB)^{2}$ and describe its special fiber $\overline{\rmY}(\rmB)$ over $\FF_{p^{2}}$. We will make the following constructions and definitions:
\begin{enumerate}
\item Let $?\in \{\circ, \bullet\}$, we put $\overline{\rmX}^{?}=\PP^{1}(\rmZ^{?}(\overline{\rmB}))$.  Let $\overline{\rmX}^{\dagger}=\rmZ_{\Iw(p)}(\overline{\rmB})$.
\item For $(?_{0}?_{1})\in \{\circ, \bullet, \dagger\}^{2}$, we put $\overline{\rmQ}^{?_{0}?_{1}}=\overline{\rmX}^{?_{0}}\times \overline{\rmX}^{?_{1}}$. 
\item For $(?_{0}?_{1})\in \{\circ, \bullet\}^{2}$, then we note that
\begin{equation*}
\overline{\rmQ}^{?_{0}?_{1}}=\PP^{1}(\rmZ^{?_{0}}(\overline{\rmB}))\times \PP^{1}(\rmZ^{?_{1}}(\overline{\rmB}))
\end{equation*}
are the irreducible components of the special fiber $\overline{\rmX}(\rmB)^{2}$ of  $\interX(\rmB)^{2}$.
\item Let $\sigma: \mathfrak{Y}(\rmB)\rightarrow \interX(\rmB)^{2}$ be the blow-up of $\mathfrak{X}(\rmB)^{2}$ along the closed subscheme $\overline{\rmQ}^{\circ\circ}$. For $(?_{0}?_{1})\in \{\circ, \bullet, \dagger\}^{2}$, let $\overline{\rmY}^{?_{0}?_{1}}$ be the strict transform of $\overline{\rmQ}^{?_{0}?_{1}}$ under $\sigma$. 
\item The dual reduction graph of $\mathfrak{Y}(\rmB)$ is given by
\begin{equation*}
\begin{tikzcd}
\overline{\rmY}^{\bullet\circ} \arrow[dd, "\overline{\rmY}^{\bullet\dagger}"', no head] \arrow[rr, "\overline{\rmY}^{\dagger\circ}", no head]        &  & \overline{\rmY}^{\circ\circ} \arrow[dd, "\overline{\rmY}^{\circ\dagger}", no head] \\
                                                       &  &                            \\
\overline{\rmY}^{\bullet\bullet} \arrow[rr, "\overline{\rmY}^{\dagger\bullet}"', no head] \arrow[rruu, "\overline{\rmY}^{\dagger\dagger}", no head] &  & \overline{\rmY}^{\circ\bullet}                                         
\end{tikzcd}
\end{equation*}
so that we have 
\begin{equation*}
\overline{\rmY}^{(0)}(\rmB)=\overline{\rmY}^{\circ\circ}\sqcup \overline{\rmY}^{\bullet\circ}\sqcup \overline{\rmY}^{\circ\bullet}\sqcup \overline{\rmY}^{\bullet\bullet}
\end{equation*}
and 
\begin{equation*}
\overline{\rmY}^{(1)}(\rmB)=\overline{\rmY}^{\circ\dagger}\sqcup \overline{\rmY}^{\bullet\dagger}\sqcup \overline{\rmY}^{\dagger\bullet}\sqcup \overline{\rmY}^{\dagger\circ}\sqcup \overline{\rmY}^{\dagger\dagger}
\end{equation*}
are the codimension $0$ and codimension $1$ strata of $\overline{\rmY}(\rmB)$.
\end{enumerate}

\begin{lemma}
The scheme $\mathfrak{Y}(\rmB)$ is regular with semi-stable reduction. We have the following descriptions of the strata of $\overline{\rmY}(\rmB)$.
\begin{enumerate}
\item The map $\sigma: \overline{\rmY}^{?_{0}?_{1}}\rightarrow \overline{\rmQ}^{?_{0} ?_{1}}$ for $?_{0}\neq ?_{1}$ is an isomorphism;
\item The map $\sigma: \overline{\rmY}^{?_{0}?_{1}}\rightarrow \overline{\rmQ}^{?_{0}?_{1}}$ for $(?_{0}?_{1})\in \{(\circ\circ), (\bullet\bullet)\}$ is the blow-up of $\overline{\rmQ}^{?_{0}?_{1}}$ along the stratum $\overline{\rmQ}^{\dagger\dagger}$;
\item The stratum $\overline{\rmY}^{\dagger\dagger}$ is a $\PP^{1}$-bundle over $\overline{\rmQ}^{\dagger\dagger}$.
\end{enumerate}
\end{lemma}
\begin{proof}
This result is well-known following from standard computations of the blow-up.
\end{proof}

\begin{corollary}\label{middle-deg}
The cohomology $\rmH^{2}(\overline{\rmY}^{(0)}(\rmB), \ZZ_{l}(1))$ is given by
\begin{equation*}
\bigoplus\limits_{(?_{0}?_{1})\in \{\circ, \bullet, \dagger\}^{2}}\rmH^{2}(\overline{\rmQ}^{?_{0}?_{1}}, \ZZ_{l}(1))\oplus \rmH^{0}(\rmQ^{\dagger\dagger},\ZZ_{l})^{\oplus2}.
\end{equation*}
\end{corollary}
\begin{proof}
This follows from the previous lemma immediately as 
\begin{equation*}
\rmH^{2}(\overline{\rmY}^{?_{0}?_{1}}, \ZZ_{l}(1))=\rmH^{2}(\overline{\rmQ}^{?_{0}?_{1}}, \ZZ_{l}(1))\oplus \rmH^{0}(\overline{\rmQ}^{\dagger\dagger}, \ZZ_{l})
\end{equation*} 
for $(?_{0}?_{1})\in\{(\circ\circ), (\bullet\bullet)\}$ and 
\begin{equation*}
\rmH^{2}(\overline{\rmY}^{?_{0}?_{1}}, \ZZ_{l}(1))=\rmH^{2}(\overline{\rmQ}^{?_{0}?_{1}}, \ZZ_{l}(1))
\end{equation*} 
for $(?_{0}?_{1})\in\{(\circ\bullet), (\bullet\circ)\}$.
\end{proof}

\subsection{Construction of the Flach elements}
We will construct elements in the motivic cohomology of the Shimura surface $\rmX(\rmB)^{2}$ given by the product of a pair of Shimura curves. Let $\theta: \rmX(\rmB)\rightarrow \rmX(\rmB)^{2}$ be the diagonal embedding and $\theta_{\ast}\rmX(\rmB)$ be the image of $\theta$. Recall $p$ is a prime dividing the discriminant of $\rmB$. We consider the element 
\begin{equation*}
\Theta^{[p]}(\rmB)=(\theta_{\ast}\rmX(\rmB), p)
\end{equation*}
which clearly lies in the kernel of $d_{1}$ in \eqref{M-complex} and therefore defines an element $\Theta^{[p]}(\rmB)$ in the motivic cohomology group $\rmH^{3}_{\calM}(\rmX(\rmB)^{2},\ZZ(2))$ of $\rmX(\rmB)^{2}$ which we will refer to it as the \emph{Flach element}. Recall that we have a divisor map for each prime $v \nmid p\rmN$ 
\begin{equation}
\mathrm{div}_{v}: \rmH^{3}_{\calM}(\rmX(\rmB)^{2},\ZZ(2))\rightarrow \rmH^{2}_{\calM}(\interX(\rmB)^{2}\otimes\FF_{v},\ZZ(1)).
\end{equation}
Note that  for each $v$ as above, $\mathrm{div}_{v}(\Theta^{[p]}(\rmB))=0$ as an element in $\rmH^{2}_{\calM}(\interX(\rmB)^{2}\otimes\FF_{v},\ZZ(1))$, this is clear since $p$ is invertible on $\interX(\rmB)^{2}\otimes\FF_{v}$. On the other hand, we also have a divisor map for the prime $p$ using the regular model $\frakY(\rmB)$
\begin{equation}
\mathrm{div}_{p}: \rmH^{3}_{\calM}(\rmX(\rmB)^{2},\ZZ(2))\rightarrow \rmH^{2}_{\calM}(\overline{\rmY}(\rmB)\otimes\FF_{p^{2}},\ZZ(1)).
\end{equation}
\begin{remark}
This element is the analogue of the Flach element used in \cite{Flach} to bound the Selmer group of the symmetric square of the Tate module of an elliptic curve. There one has to choose carefully a special Siegel modular unit while here our construction is much simpler but fits our purpose.
\end{remark}

\subsection{Reciprocity law for the Flach element}
Recall the setting we are in. We have a modular form $f\in \rmS^{\new}_{2}(\rmN)$. Such form $f$ admits a Jacquet--Langlands transfer $f^{\dagger}$ to a modular form in $\Gamma(\rmZ(\overline{\rmB}), \calO)$. Let $p$ be an $n$-admissible prime for $f$. We let $\TT=\TT_{\rmN^{+}, \rmN^{-}}$ (resp.  $\TT^{[p]}=\TT_{\rmN^{+}, p\rmN^{-}}$) be the $\lambda$-adic Hecke algebra corresponding to the cusp forms of level $\rmN=\rmN^{+}\rmN^{-}$ (resp. of level $\rmN p=\rmN^{+}\rmN^{-}p$) which is new at primes dividing $\rmN^{-}$ (resp. at primes dividing $p\rmN^{-}$). 
\begin{enumerate}
\item Since $f$ is an eigenform, we have a morphism
$\phi_{f}: \TT\rightarrow \calO$
corresponding to the system of Hecke eigenvalues of $f$. More precisely, we have  $\phi_{f}(\rmT_{v})=a_{v}(f)$ for $v\nmid \rmN$  and  $\phi_{f}(\rmU_{v})=a_{v}(f)$ for $v\mid \rmN$. Let $\phi_{f, n}: \TT\rightarrow \calO_{n}$ be reduction of the map $\phi_{f}: \TT\rightarrow \calO$ modulo $\varpi^{n}$. We denote by $\mathfrak{p}_{f, n}$  the kernel of this map and $\mathfrak{m}_{f}$ the unique maximal ideal of $\TT$ containing $\frakp_{f, n}$. We always assume that $\frakm=\frakm_{f}$ is absolutely irreducible and write $\frakp_{n}=\frakp_{f, n}$. 
\item Since $p$ is $n$-admissible for $f$, Theorem \ref{level-raise-curve} provides a surjective homomorphism $\phi^{[p]}_{f, n}: \TT^{[p]}\rightarrow \calO_{n}$ such that $\phi^{[p]}_{f, n}$ agrees with $\phi_{f, n}$ at all Hecke operators away from $p$ and sends $\rmU_{p}$ to $\epsilon_{p}(f)$. We will denote by $\frakp^{[p]}_{ n}=\frakp^{[p]}_{f, n}$ the kernel of $\phi^{[p]}_{f, n}$ and $\frakm^{[p]}=\frakm^{[p]}_{f}$ the corresponding maximal ideal. 
\end{enumerate}

The product of Hecke algebras $\TT^{[p]}\times \TT^{[p]}$ acts naturally on the surface $\rmX(\rmB)^{2}$ by correspondences. And by the K\"unneth decomposition, it induces an action of  $\TT^{[p]\otimes 2}$ on the cohomology
$\rmH^{2}(\rmX(\rmB)^{2}\otimes\overline{\QQ}, \calO(2))$. Let $\underline{\frakm}^{[p]}$ be the maximal ideal of $\TT^{[p]\otimes 2}$ given by the pair $(\frakm^{[p]}, \frakm^{[p]})$.
\begin{lemma} 
Suppose that $p$ is an $n$-admissible prime for $f$ and that $\frakm$ is absolutely irreducible.
\begin{enumerate}\label{Hochschild}
\item We have an isomorphism
\begin{equation*}
\rmH^{2}(\rmX(\rmB)^{2}\otimes\overline{\QQ}, \calO(2))_{\underline{\frakm}^{[p]}}\cong \rmH^{1}(\rmX(\rmB)\otimes\overline{\QQ},\calO(1))^{\otimes2}_{\frakm^{[p]}}.
\end{equation*}
\item We have an isomorphism
\begin{equation*}
\rmH^{3}(\rmX(\rmB)^{2}, \calO(2))_{\underline{\frakm}^{[p]}}\cong\rmH^{1}(\QQ, \rmH^{1}(\rmX(\rmB)\otimes\overline{\QQ},\calO(1))^{\otimes2}_{\frakm^{[p]}}).
\end{equation*}
\end{enumerate}
\end{lemma}
\begin{proof}
The first part follows from the irreducibility of $\frakm$ and hence that of $\frakm^{[p]}$ and the fact that the $\rmH^{0}$ and $\rmH^{2}$ of the Shimura curve $\rmX(\rmB)$ are Eisenstein. The second part follows from the Hochschild--Serre spectral sequence for $\rmH^{3}(\rmX(\rmB)^{2}, \calO(2))_{\underline{\frakm}^{[p]}}$ and the first part.
\end{proof}

We consider the Chern character map given in \eqref{chern} localized at $\frakm^{[p]}$
\begin{equation*}
\mathrm{ch}_{\frakm^{[p]}}: \rmH^{3}_{\calM}(\rmX(\rmB)^{2}, \ZZ(2))\rightarrow \rmH^{3}(\rmX(\rmB)^{2}, \calO(2))_{\underline{\frakm}^{[p]}}.
\end{equation*}
From the part $(2)$ of the above lemma, this map can be rewritten as
\begin{equation*}
\mathrm{AJ}_{\frakm^{[p]}}: \rmH^{3}_{\calM}(\rmX(\rmB)^{2}, \ZZ(2))\rightarrow \rmH^{1}(\QQ, \rmH^{1}(\rmX(\rmB)\otimes\overline{\QQ},\calO(1))^{\otimes2}_{\frakm^{[p]}})
\end{equation*}
and this map is usually referred to as the Abel--Jacobi map. We will call the element
\begin{equation*}
\kappa^{[p]}=\mathrm{AJ}_{\frakm^{[p]}}(\mathrm{\Theta^{[p]}(\rmB)}) 
\end{equation*} 
the Flach class. The goal for the rest of this subsection is to analyze the local behaviour of $\kappa^{[p]}$ at each prime $v$ of $\QQ$. 

\begin{lemma}
Suppose that $v\nmid p\rmN$. Then we have
\begin{equation*}
\mathrm{res}_{v}(\kappa^{[p]})\in \rmH^{1}_{\mathrm{fin}}(\QQ_{v}, \rmH^{1}(\rmX(\rmB)\otimes\overline{\QQ},\calO(1))^{\otimes2}_{\frakm^{[p]}}).
\end{equation*}
\end{lemma}
\begin{proof}
Since $\rmH^{1}(\rmX(\rmB)\otimes\overline{\QQ},\calO(1))^{\otimes2}_{\frakm^{[p]}}$ is unramified at $v$ as a Galois module, we have
\begin{equation*}
\begin{aligned}
\rmH^{1}_{\mathrm{sin}}(\QQ_{v}, \rmH^{1}(\rmX(\rmB)\otimes\overline{\QQ},\calO(1))^{\otimes2}_{\frakm^{[p]}})
&=(\rmH^{2}(\interX(\rmB)^{2}\otimes\overline{\FF}_{v}, \calO(1))_{\underline{\frakm}^{[p]}})^{\rmG_{\FF_{v}}} \\
&=(\rmH^{1}(\mathfrak{X}(\rmB)\otimes\overline{\FF}_{v},\calO(1))^{\otimes2}_{\frakm^{[p]}}(-1))^{\rmG_{\FF_{v}}}.\\
\end{aligned}
\end{equation*}
On the other hand, we have a commutative diagram by \cite[Theorem 3.1.1]{Weston1}
\begin{equation*}
\begin{tikzcd}
\rmH^{3}_{\calM}(\rmX(\rmB)^{2}\otimes\QQ_{v},\ZZ(2)) \arrow[r, "\mathrm{ch}_{\frakm^{[p]}}"] \arrow[d, "\mathrm{div}_{v}"'] & \rmH^{3}(\rmX(\rmB)^{2}\otimes\QQ_{v}, \calO(2))_{\underline{\frakm}^{[p]}} \arrow[d, "\partial_{v}"] \\
 \rmH^{2}_{\calM}(\interX(\rmB)^{2}\otimes\FF_{v},\ZZ(1)) \arrow[r]                 &      (\rmH^{2}(\interX(\rmB)^{2}\otimes\overline{\FF}_{v}, \calO(1))_{\underline{\frakm}^{[p]}})^{\rmG_{\FF_{v}}}    
\end{tikzcd}
\end{equation*}
where the bottom map is the usual cycle class map for $\interX(\rmB)^{2}\otimes\FF_{v}$.

Note that we have an identification
\begin{equation*}
\rmH^{3}(\rmX(\rmB)^{2}\otimes\QQ_{v}, \calO(2))_{\underline{\frakm}^{[p]}}=\rmH^{1}(\QQ_{v}, \rmH^{1}(\rmX(\rmB)\otimes\overline{\QQ}_{v},\calO(1))^{\otimes2}_{\frakm^{[p]}})
\end{equation*}
by Lemma \ref{Hochschild} and an identification
\begin{equation*}
(\rmH^{2}(\interX(\rmB)^{2}\otimes\overline{\FF}_{v}, \calO(1))_{\underline{\frakm}^{[p]}})^{\rmG_{\FF_{v}}}=\rmH^{1}_{\mathrm{sin}}(\QQ_{v}, \rmH^{1}(\rmX(\rmB)\otimes\overline{\QQ}_{v},\calO(1))^{\otimes2}_{\frakm^{[p]}})
\end{equation*}
by the unramifiedness of $\rmH^{1}(\rmX(\rmB)\otimes\overline{\QQ}_{v},\calO(1))^{\otimes2}_{\frakm^{[p]}}$. Then the right vertical map $\partial_{v}$ is given by the singular quotient map
\begin{equation*}
\rmH^{1}(\QQ_{v}, \rmH^{1}(\rmX(\rmB)\otimes\overline{\QQ}_{v},\calO(1))^{\otimes2}_{\frakm^{[p]}})\rightarrow \rmH^{1}_{\mathrm{sin}}(\QQ_{v}, \rmH^{1}(\rmX(\rmB)\otimes\overline{\QQ}_{v},\calO(1))^{\otimes2}_{\frakm^{[p]}})
\end{equation*} 
using these two identifications. Since $\mathrm{div}_{v}(\Theta^{[p]}(\rmB))=0$ as an element in $\rmH^{2}_{\calM}(\interX(\rmB)^{2}\otimes\FF_{v},\ZZ(1))$, it follows immediately from the above commutative diagram that
\begin{equation*}
\mathrm{res}_{v}(\kappa^{[p]})\in \rmH^{1}_{\mathrm{fin}}(\QQ_{v}, \rmH^{1}(\rmX(\rmB)\otimes\overline{\QQ}_{v},\calO(1))^{\otimes2}_{\frakm^{[p]}}).
\end{equation*}
\end{proof}

Next we analyze the local behaviour of $\kappa^{[p]}$ at $p$. First, we prove the following proposition which is the analog of the arithmetic level raising theorem in Theorem \ref{level-raise-curve} for the surface $\rmX(\rmB)^{2}$.
\begin{proposition}\label{level-raise-prod}
Suppose that $p$ is an $n$-admissible prime for $f$ and $\frakm$ is absolutely irreducible. Then we have an isomorphism
\begin{equation*}
\rmH^{1}_{\mathrm{sin}}(\QQ_{p}, \rmH^{1}(\rmX(\rmB)\otimes\overline{\QQ}_{p},\calO_{n}(1))^{\otimes2}_{\frakm^{[p]}})
\cong \bigoplus^{2}_{i=1}\Gamma(\rmZ(\overline{\rmB}),\calO_{n})^{\otimes 2}_{\frakm}.
\end{equation*}
\end{proposition}
\begin{proof}
Since $\rmH^{1}(\rmX(\rmB)\otimes\overline{\QQ}_{p},\calO_{n}(1))$ is unramified at $p$ by the definition of an $n$-admissible prime for $f$, we have
\begin{equation*}
\rmH^{1}_{\mathrm{sin}}(\QQ_{p^{2}}, \rmH^{1}(\rmX(\rmB)\otimes\overline{\QQ}_{p},\calO_{n}(1))^{\otimes2}_{\frakm^{[p]}})\cong(\rmH^{1}(\rmX(\rmB)\otimes\overline{\QQ}_{p},\calO_{n}(1))^{\otimes2}_{\frakm^{[p]}}(-1))^{\rmG_{\FF_{p^{2}}}}.
\end{equation*}
By the arithmetic level raising theorem for the Shimura curve $\rmX(\rmB)$ as in Theorem \ref{level-raise-curve} and its consequence in Proposition \ref{split}, we have 
\begin{equation*}
(\rmH^{1}(\rmX(\rmB)\otimes\overline{\QQ}_{p},\calO_{n}(1))^{\otimes2}_{\frakm^{[p]}}(-1))^{\rmG_{\FF_{p^{2}}}}\cong \bigoplus^{2}_{i=1}\Gamma(\rmZ(\overline{\rmB}),\calO_{n})^{\otimes 2}_{\frakm}.
\end{equation*}
Since the non-trivial element in $\mathrm{Gal}(\FF_{p^{2}}/\FF_{p})$ acts on the right-hand side by the product of $\rmU_{p}$-eigenvalues which is $\epsilon^{2}_{p}(f)=1$, this isomorphism therefore descends to to an isomorphism
\begin{equation*}
(\rmH^{1}(\rmX(\rmB)\otimes\overline{\QQ}_{p},\calO_{n}(1))^{\otimes2}_{\frakm^{[p]}}(-1))^{\rmG_{\FF_{p}}}\cong \bigoplus^{2}_{i=1}\Gamma(\rmZ(\overline{\rmB}),\calO_{n})^{\otimes 2}_{\frakm}.
\end{equation*}
\end{proof}
\begin{remark}\label{sin-remark}
The identification of 
\begin{equation*}
\bigoplus^{2}_{i=1}\Gamma(\rmZ(\overline{\rmB}),\calO_{n})^{\otimes 2}_{\frakm}\cong\rmH^{1}_{\mathrm{sin}}(\QQ_{p}, \rmH^{1}(\rmX(\rmB)\otimes\overline{\QQ}_{p},\calO_{n}(1))^{\otimes2}_{\frakm^{[p]}}).
\end{equation*}
is really induced by the canonical map
\begin{equation*}
\rmH^{2}(\overline{\rmY}^{(0)}(\rmB)\otimes\overline{\FF}_{p}, \calO_{n}(1))_{\underline{\frakm}^{[p]}}\rightarrow \rmH^{2}({\rmX}(\rmB)\otimes\overline{\QQ}_{p}, \calO_{n}(1))_{\underline{\frakm}^{[p]}}
\end{equation*}
by taking $\rmG_{\FF_{p^{2}}}$-invariant part on both sides. Indeed, one readily finds that 
\begin{equation*}
\rmH^{2}(\overline{\rmY}^{(0)}(\rmB)\otimes\overline{\FF}_{p}, \calO_{n}(1)) 
\end{equation*}
is given by the direct sum of 
\begin{equation*}
\begin{aligned}
\bigoplus\limits_{(?_{0}?_{1})\in \{\circ, \bullet\}^{2}}&\bigoplus\limits_{(i_{0},i_{1})\in\{0,2\}}\rmH^{i_{0}}(\PP^{1}(\rmZ^{?_{0}}(\overline{\rmB}), \calO_{n}(i_{0}/2))\otimes \rmH^{i_{1}}(\PP^{1}(\rmZ^{?_{1}}(\overline{\rmB}), \calO_{n}(i_{1}/{2}))\\
\end{aligned}
\end{equation*}
and $\rmH^{0}(\rmQ^{\dagger\dagger},\calO_{n})^{\oplus2}$ by Corollary \ref{middle-deg}. Then our assertion follows from Remark \ref{sub-quotient}: the space $\bigoplus^{2}\limits_{i=1}\Gamma(\rmZ(\overline{\rmB}),\calO_{n})^{\otimes 2}_{\frakm}$ is identified with a quotient of 
\begin{equation*}
\begin{aligned}
\bigoplus\limits_{(?_{0}?_{1})\in \{\circ, \bullet\}^{2}}&\bigoplus\limits_{(i_{0},i_{1})\in\{0,2\}}\rmH^{i_{0}}(\PP^{1}(\rmZ^{?_{0}}(\overline{\rmB}), \calO_{n}(i_{0}/2))\otimes \rmH^{i_{1}}(\PP^{1}(\rmZ^{?_{1}}(\overline{\rmB}), \calO_{n}(i_{1}/{2}))\\
\end{aligned}
\end{equation*}
while $\rmH^{1}_{\mathrm{sin}}(\QQ_{p}, \rmH^{1}(\rmX(\rmB)\otimes\overline{\QQ}_{p},\calO_{n}(1))^{\otimes2}_{\frakm^{[p]}})$ is identified with $(\rmH^{2}({\rmX}(\rmB)\otimes\overline{\QQ}_{p}, \calO_{n}(1)))^{\rmG_{\FF_{p^{2}}}}_{\underline{\frakm}^{[p]}}$.
\end{remark}

Let ${\partial}_{p}(\kappa^{[p]})$ be the singular residue of the class $\mathrm{res}_{p}(\kappa^{[p]})$. In light of Proposition \ref{level-raise-prod}, the element $\partial_{p}(\kappa^{[p]})$ can be viewed as an element of $\bigoplus^{2}\limits_{i=1}\Gamma(\rmZ(\overline{\rmB}),\calO_{n})^{\otimes 2}_{\frakm}$. Let $\partial^{(i)}_{p}(\kappa^{[p]})$ be the projection of $\partial_{p}(\kappa^{[p]})$ to the $i$-th copy of  $\bigoplus^{2}\limits_{i=1}\Gamma(\rmZ(\overline{\rmB}),\calO_{n})^{\otimes 2}_{\frakm}$ under this identification for $i=1,2$. 

We define a bilinear pairing 
\begin{equation*}
(\cdot,\cdot): \Gamma(\rmZ(\overline{\rmB}),\calO)^{\otimes 2}\times \Gamma(\rmZ(\overline{\rmB}),\rmE_{\lambda}/\calO)^{\otimes 2}\rightarrow \calO
\end{equation*}
by the formula
\begin{equation*}
(\otimes^{2}_{i=1}\zeta_{i}, \otimes^{2}_{i=1}\phi_{i})= \sum_{(z_{1}, z_{2})\in \rmZ(\overline{\rmB})^{2}} \zeta_{1}\phi_{1}(z_{1})\times\zeta_{2}\phi_{2}(z_{2})
\end{equation*}
for any $\otimes^{2}_{i=1}\zeta_{i}\in \Gamma(\rmZ(\overline{\rmB}), \calO)^{\otimes 2}$ and $ \otimes^{2}_{i=1}\phi_{i}\in \Gamma(\rmZ(\overline{\rmB}), \rmE_{\lambda}/\calO)^{\otimes2}$. The above paring induces 
\begin{equation*}
 (\cdot, \cdot): \Gamma(\rmZ(\overline{\rmB}), \calO)^{\otimes2}_{/\frakp_{n}}\times  \Gamma(\rmZ(\overline{\rmB}), \rmE_{\lambda}/\calO)^{\otimes2}[\frakp_{n}]\rightarrow \calO_{n}
\end{equation*}
which can be rewritten as 
\begin{equation*}
 (\cdot, \cdot): \Gamma(\rmZ(\overline{\rmB}), \calO_{n})_{\frakm}^{\otimes2}\times  \Gamma(\rmZ(\overline{\rmB}), \calO_{n})^{\otimes2}_{\frakm}\rightarrow \calO_{n}.
\end{equation*}
The following is the main result of this subsection which we will refer to as the \emph{reciprocity law for the Flach element}.

\begin{theorem}\label{reciprocity}
Let $p$ be an $n$-admissible prime for $f$ and $\frakm$ be absolutely irreducible. Then the following formula 
\begin{equation*}
(\partial^{(i)}_{p}(\kappa^{[p]}), \phi_{1}\otimes\phi_{2})=\sum_{z\in \rmZ(\overline{\rmB})}\phi_{1}\phi_{2}(z)
\end{equation*}
holds for any $\phi_{1}\otimes\phi_{2}\in \Gamma(\rmZ(\overline{\rmB}),\calO_{n})^{\otimes2}_{\frakm}$ and any $i=1,2$.
\end{theorem}

\begin{proof}
The proof of this will follow from the commutative diagram below 
\begin{equation*}
\begin{tikzcd}
\rmH^{3}_{\calM}(\rmX(\rmB)^{2}\otimes\QQ_{p^{2}}, \ZZ(2)) \arrow[rr, "\mathrm{ch}_{\frakm^{[p]}, n}"] \arrow[ddd, "\mathrm{div}_{p}"'] &  & \rmH^{1}(\QQ_{p^{2}}, \rmH^{1}(\rmX(\rmB)\otimes\overline{\QQ}_{p},\calO_{n}(1))^{\otimes2}_{\frakm^{[p]}}) \arrow[d, "\partial_{p}"] \\
                                            &   & \rmH^{1}_{\mathrm{sin}}(\QQ_{p^{2}}, \rmH^{1}(\rmX(\rmB)\otimes\overline{\QQ}_{p},\calO_{n}(1))^{\otimes2}_{\frakm^{[p]}}) \\
                                            &    &  \bigoplus^{2}_{i=1}\Gamma(\rmZ(\overline{\rmB}),\calO_{n})^{\otimes 2}_{\frakm} \arrow[u,"\alpha_{3}"] \\
\rmH^{2}_{\calM}(\overline{\rmY}(\rmB)\otimes\FF_{p^{2}}, \ZZ(1))    \arrow[rr, "\alpha_{1}"] & &\rmH^{2}(\overline{\rmY}^{(0)}(\rmB)\otimes\overline{\FF}_{p}, \calO_{n}(1))_{\underline{\frakm}^{[p]}}^{\rmG_{\FF_{p^{2}}}}  \arrow[u, "\alpha_{2}"]
\end{tikzcd}
\end{equation*}
where the non-obvious maps are given by
\begin{enumerate}
\item $\alpha_{1}$ is the cycle class map 
\begin{equation*}
\rmH^{2}_{\calM}(\overline{\rmY}(\rmB)\otimes\FF_{p^{2}}, \ZZ(1))\rightarrow \rmH^{2}(\overline{\rmY}(\rmB)\otimes\overline{\FF}_{p}, \calO_{n}(1))_{\underline{\frakm}^{[p]}}^{\rmG_{\FF_{p^{2}}}} 
\end{equation*}
followed by the natural restriction map 
\begin{equation*}
\rmH^{2}(\overline{\rmY}(\rmB)\otimes\overline{\FF}_{p}, \calO_{n}(1))_{\underline{\frakm}^{[p]}}^{\rmG_{\FF_{p^{2}}}}\rightarrow \rmH^{2}(\overline{\rmY}^{(0)}(\rmB)\otimes\overline{\FF}_{p}, \calO_{n}(1))_{\underline{\frakm}^{[p]}}^{\rmG_{\FF_{p^{2}}}};
\end{equation*}
\item $\alpha_{2}$ is given by the following: recall that by Corollary \ref{middle-deg} the cohomology group 
\begin{equation*}
\rmH^{2}(\overline{\rmY}^{(0)}(\rmB), \calO_{n}(1))
\end{equation*}
is given by the direct sum of 
\begin{equation*}
\begin{aligned}
\bigoplus\limits_{(?_{0}?_{1})\in \{\circ, \bullet\}^{2}}&\bigoplus\limits_{(i_{0},i_{1})\in\{0,2\}}\rmH^{i_{0}}(\PP^{1}(\rmZ^{?_{0}}(\overline{\rmB}), \calO_{n}(i_{0}/2))\otimes \rmH^{i_{1}}(\PP^{1}(\rmZ^{?_{1}}(\overline{\rmB}), \calO_{n}(i_{1}/{2}))\\
\end{aligned}
\end{equation*}
and 
\begin{equation*}
\rmH^{0}(\rmQ^{\dagger\dagger},\calO_{n})^{\oplus2}
\end{equation*}
and $\alpha_{2}$ is the map from the direct factor 
\begin{equation*}
\bigoplus\limits_{(?_{0}?_{1})\in \{\circ, \bullet\}^{2}}\bigoplus\limits_{(i_{0},i_{1})\in\{0,2\}}\rmH^{i_{0}}(\PP^{1}(\rmZ^{?_{0}}(\overline{\rmB}), \calO_{n}(i_{0}/2))\otimes \rmH^{i_{1}}(\PP^{1}(\rmZ^{?_{1}}(\overline{\rmB}), \calO_{n}(i_{1}/2)) 
\end{equation*}
to $\bigoplus^{2}\limits_{i=1}\Gamma(\rmZ(\overline{\rmB}),\calO_{n})^{\otimes 2}_{\frakm}$.

\item $\alpha_{3}$ is the arithmetic level raising isomorphism in Proposition \ref{level-raise-prod}.
\end{enumerate}
The proof for the commutativity of the above diagram is the same as that of \cite[Theorem 3.1.1]{Weston1} using the regular pair $(\overline{\rmY}(\rmB)^{(0)}, \frakY(\rmB))$ instead of a smooth pair and  taking into account of Remark \ref{sub-quotient} and Remark \ref{sin-remark}. From the commutativity of this diagram, the class $\partial^{(i)}_{p}(\kappa^{[p]})$ in $\Gamma(\rmZ(\overline{\rmB}),\calO_{n})^{\otimes 2}_{\frakm} $ is given by the function $\vartheta_{\ast}\mathbf{1}_{\overline{\rmB}}$ where $\mathbf{1}_{\overline{\rmB}}$  is the characteristic function  of the Shimura set $\rmZ(\overline{\rmB})$ and where $\vartheta: \rmZ(\overline{\rmB})\rightarrow \rmZ(\overline{\rmB})^{2}$ is the diagonal embedding of the Shimura sets. It follows then that
\begin{equation*}
\begin{aligned}
(\partial^{(i)}_{p}(\kappa^{[p]}), \phi_{1}\otimes\phi_{2})&=(\nu_{\ast}\mathbf{1}_{\overline{\rmB}}, \phi_{1}\otimes\phi_{2})\\
&=\sum\limits_{z\in\rmZ(\overline{\rmB})}\phi_{1}(z)\phi_{2}(z).
\end{aligned}
\end{equation*}
\end{proof}

\begin{remark}\label{depth}
Since $\rmN^{-}$ is square free and consists of odd number of prime factors, $f\in\rmS^{\new}_{2}(\Gamma_{0}(\rmN))$ admits a Jacquet--Langlands transfer to a quaternionic modular form $f^{\dagger}\in \Gamma(\rmZ(\overline{\rmB}),\calO)_{\frakm}$. We will insist that $f^{\dagger}$ is normalized. This means that there is a point $z\in\rmZ(\overline{\rmB})$ such that $f^{\dagger}(z)\mod \lambda \neq 0$. Then the above proposition implies that
\begin{equation*}
(\partial^{(i)}_{p}(\kappa^{[p]}), f^{\dagger}\otimes f^{\dagger})=\langle f^{\dagger}, f^{\dagger}\rangle
\end{equation*} 
where the righthand side is the Petersson norm of $f^{\dagger}$. We will write $\calP(f^{\dagger})=\langle f^{\dagger}, f^{\dagger}\rangle$ and call it the {\em quaternionic period of} of $f$. 
\end{remark}

\section{Bounding the adjoint Selmer groups}
\subsection{Generalities on Selmer groups}
We will abuse the notation and consider a general Galois representation $\rho: \rmG_{\QQ}\rightarrow\GL(\rmV)$ over $\rmE_{\lambda}$. Suppose that $\rmT$ is a stable Galois $\calO$-lattice in $\rmV$ and set $\calM=\rmV/\rmT$.  These Galois modules fit in the exact sequence 
\begin{equation*}
0\rightarrow \rmT\xrightarrow{i} \rmV\xrightarrow{\mathrm{pr}} \calM\rightarrow 0.
\end{equation*}
Let $\rmM_{n}=\calM[\lambda^{n}]$ and $\rmT_{n}=\rmT/\lambda^{n}$. Let $i_{n}: \rmM_{n}\hookrightarrow \rmM$ and $\mathrm{pr}_{n}:\rmT\rightarrow \rmT_{n}$ be the natural inclusion and projection map.

We recall the definitions concerning the local conditions defining the Bloch--Kato Selmer group for $\calM$ and $\rmM_{n}$:
\begin{enumerate}
\item For $v\neq l$, we define $\rmH^{1}_{f}(\QQ_{v},\rmV)=\rmH^{1}_{\mathrm{fin}}(\QQ_{v}, \rmV)$;
\item For $v=l$, we define $\rmH^{1}_{f}(\QQ_{p},\rmV)=\mathrm{ker}(\rmH^{1}(\QQ_{p}, \rmV)\rightarrow \rmH^{1}(\QQ_{p}, \rmV\otimes \rmB_{\mathrm{cris}}))$;
\item We define $\rmH^{1}_{f}(\QQ_{v},\calM)=\mathrm{pr}_{\ast}\rmH^{1}_{f}(\QQ_{v}, \rmV)$ and $\rmH^{1}_{f}(\QQ_{v},\rmM_{n})=i^{\ast}_{n}\rmH^{1}_{f}(\QQ_{v}, \calM)$  for each $v$.
\end{enumerate}
Then we define the Bloch--Kato Selmer group of $\calM$ by
\begin{equation*}
\rmH^{1}_{f}(\QQ, \calM)=\mathrm{ker}\{\rmH^{1}(\QQ,\calM)\rightarrow \prod_{v}\frac{\rmH^{1}(\QQ_{v}, \calM)}{\rmH^{1}_{f}(\QQ_{v},\calM)}\} 
\end{equation*}
for $\calM$. We also define the Bloch--Kato Selmer group $\rmH^{1}_{f}(\QQ, \rmM_{n})$ of $\rmM_{n}$ in the same way. Moreover we have
\begin{equation*}
\rmH^{1}_{f}(\QQ, \calM)=\lim\limits_{\longrightarrow n}\rmH^{1}_{f}(\QQ, \rmM_{n})
\end{equation*}
and have an exact sequence 
\begin{equation*}
0\rightarrow \mathrm{pr}_{\ast}\rmH^{1}_{f}(\QQ, \rmV)\rightarrow \rmH^{1}_{f}(\QQ, \calM)\rightarrow \Sha(\QQ, \calM)\rightarrow 0
\end{equation*}
the Tate-Shafarevich group $\Sha(\QQ, \calM)$ for $\calM$. 

\subsection{Main theorem on the adjoint representations}
Recall $f$ is a modular form of level $\rmN=\rmN^{+}\rmN^{-}$. We assume that $\rmN^{-}$ is square-free and is divisible by odd number of primes. In this case $f$ admits a Jacquet--Langlands transfer to a modular form in $\rmS^{\overline{\rmB}}(\rmN^{+})$ which we assume is normalized and gives rise to an element $f^{\dagger}$ in $\Gamma(\rmZ(\overline{\rmB}), \calO)$. Consider its Galois representation provided by the Eichler--Shimura construction
\begin{equation*}
\rho_{f,\lambda}: \rmG_{\QQ}\rightarrow \GL_{2}(\rmE_{\lambda})=\mathrm{Aut}(\rmV_{f,\lambda})
\end{equation*}
whose representation space is given by $\rmV_{\rho}=\rmV_{f,\lambda}$. We assume the residual Galois representation $\overline{\rho}=\overline{\rho}_{f,\lambda}$ is absolutely irreducible and satisfies all the conditions in Assumption \ref{ass1}. We fix a stable lattice 
$\rmT_{\rho}$ for $\rho_{f,\lambda}$ which is unique up to homothety. Let $\rmT_{n}$ be the reduction of $\rmT$ by $\lambda^{n}$ for each $n\geq 1$. Let $\calM_{\rho}=\rmV_{\rho}/\rmT_{\rho}$ and  we have an exact sequence
\begin{equation*}
0\rightarrow \rmT_{\rho}\rightarrow \rmV_{\rho}\rightarrow \calM_{\rho}\rightarrow 0.
\end{equation*}
Consider the tensor product $\rmT_{\rho}\otimes \rmT_{\rho}$ which decomposes as
\begin{equation}
\rmT_{\rho}\otimes\rmT_{\rho}=\mathrm{Sym}^{2}\rmT_{\rho}\oplus \wedge^{2}\rmT_{\rho}.
\end{equation}
Consider the representation $\mathrm{Sym}^{2}\rmT_{\ast}$ for $\ast\in\{\rho, n\}$, we know that the Cartier dual of $\mathrm{Sym}^{2}\rmT_{\rho}$ is given by $\mathrm{Sym}^{2}\rmT_{\rho}(-1)$ which is also known as $\mathrm{Ad}^{0}(\rmT_{\rho})$. All these constructions and definitions carry over when we replace $\rmT_{\rho}$ by $\calM_{\rho}$. In the following discussion we will be mostly concerned with  $\rmN=\mathrm{Sym}^{2}\rmT_{\rho}$, and $\rmN_{n}=\mathrm{Sym}^{2}\rmT_{n}$ and their Cartier duals are given by $\rmM=\Ad^{0}(\calM_{\rho})$ and $\rmM_{n}=\Ad^{0}(\calM_{\rho})[\lambda^{n}]$. The following theorem is the main step towards the final result of this article.

\begin{theorem}\label{main}
Let $f\in \rmS^{\new}_{2}(\Gamma_{0}(\rmN))$ be a newform of weight $2$ with $\rmN=\rmN^{+}\rmN^{-}$ such that $\rmN^{-}$ is squarefree and consists of odd number of prime factors. Let $f^{\dagger}$ be the normalized Jacquet--Langlands transfer of $f$ in $\Gamma(\rmZ(\overline{\rmB}),\calO)$. Let $\eta=\varpi^{\nu}$ with $\nu=\mathrm{ord}_{\lambda}(\calP(f^{\dagger}))$.
\begin{enumerate}
\item We assume that the residual Galois representation $\overline{\rho}_{f,\lambda}$ satisfies Assumption \ref{ass1}; 
\item We further assume that
\begin{equation*}
\rmH^{1}(\QQ(\rmM_{n})/\QQ, \rmM_{n})=0
\end{equation*} 
for every $n\geq 1$ and where $\QQ(\rmM_{n})$ is the splitting field of the Galois module $\rmM_{n}$. 
\end{enumerate}
Then $\eta$ annihilates the Selmer group $\rmH^{1}_{f}(\QQ, \Ad^{0}(\calM_{\rho}))$. In particular, 
\begin{equation*}
\mathrm{leng}_{\calO}\rmH^{1}_{f}(\QQ, \Ad^{0}(\calM_{\rho}))\leq \nu
\end{equation*}
\end{theorem}

\subsection{The Flach system argument} 
To prove the above theorem, we will show that under the assumptions of the theorem, $\eta$ annihilates each finite Selmer group $\rmH^{1}_{f}(\QQ, \rmM_{n})$. Here the argument follows closely that of \cite{Flach} and \cite{Weston1}. 

\begin{lemma}
Let $p$ be an $n$-admissible prime for $f$. Then the cohomology 
\begin{equation*}
\rmH^{1}(\rmX(\rmB)\otimes\overline{\QQ},\calO_{n}(1))^{\otimes2}_{\frakm^{[p]}}
\end{equation*}
is isomorphic to a direct sum of two copies of $\rmT_{n}\otimes\rmT_{n}$. In particular, there is a projection from $\rmH^{1}(\rmX(\rmB)\otimes\overline{\QQ},\calO_{n}(1))^{\otimes2}_{\frakm^{[p]}}$ to $\rmN_{n}$.
\end{lemma}
\begin{proof}
By the same argument as in \cite[Theorem 5.17]{BD-Main}, the cohomology 
$\rmH^{1}(\rmX_{\rmN^{+}}(\rmB)\otimes\overline{\QQ},\calO_{n}(1))_{\frakm^{[p]}}$
is isomorphic to $\rmT_{n}$. By the cleanness of $d$ and Lemma \ref{clean-switch} and the proof Lemma \ref{free},   $\rmH^{1}(\rmX(\rmB)\otimes\overline{\QQ},\calO_{n}(1))_{\frakm^{[p]}}$ is isomorphic to two copies of $\rmH^{1}(\rmX_{\rmN^{+}}(\rmB)\otimes\overline{\QQ},\calO_{n}(1))_{\frakm^{[p]}}$. The lemma follows .
\end{proof}
Recall, we have the global cohomology class 
\begin{equation*}
\kappa^{[p]}\in \rmH^{1}(\QQ, \rmH^{1}(\rmX(\rmB)\otimes\overline{\QQ},\calO(1))^{\otimes2}_{\frakm^{[p]}}). 
\end{equation*}
By the above lemma, we can project this class to $\rmH^{1}(\QQ, \rmN_{n})$ and we will denote by $\kappa^{[p]}_{n}$ the resulting class. We say $\kappa^{[p]}_{n}$ has singular depth $\eta$ at $p$ if the quotient $\rmH^{1}_{\sin}(\QQ_{p}, \rmN_{n})/\partial_{p}(\kappa^{[p]}_{n})$ is annihilated by $\eta$.

\begin{lemma}\label{index}
Let $p$ be an $n$-admissible prime for $f$. Then the singular part of the cohomology $\rmH^{1}_{\sin}(\QQ_{p}, \rmN_{n})$ is free of rank one
over $\calO_{n}$ and the Flach class $\kappa^{[p]}_{n}$ has singular depth $\eta$.
\end{lemma}
\begin{proof}
By the definition of an $n$-admissible prime, it is easily calculated that the Frobenius eigenvalues at $p$ for $\rmN_{n}$ are distinct and are given by $1, p, p^{2}$. Since we have 
\begin{equation*}
\rmH^{1}_{\sin}(\QQ_{p}, \rmN_{n})=\Hom(\ZZ_{l}(1),  \rmN_{n})^{\rmG_{\FF_{p}}}=(\rmN_{n}(-1))^{\rmG_{\FF_{p}}}, 
\end{equation*}
$\rmH^{1}_{\sin}(\QQ_{p}, \rmN_{n})$ is free of rank one over $\calO_{n}$.  The assertions for the singular depth follows from Remark \ref{depth} after the reciprocity law proved in Proposition \ref{reciprocity}. 
\end{proof}

\begin{lemma}\label{lm1}
Let $p$ be an $n$-admissible prime for $f$ and define 
\begin{equation*}
\rmH^{1}_{\{p\}}(\QQ,\rmM_{n})=\mathrm{ker}\{\rmH^{1}(\QQ,\rmM_{n})\rightarrow \rmH^{1}(\QQ_{p},\rmM_{n})\}.
\end{equation*}
Then we have 
\begin{equation*}
\eta\rmH^{1}_{f}(\QQ, \rmM_{n})\subset \rmH^{1}_{\{p\}}(\QQ,\rmM_{n}).
\end{equation*}
\end{lemma}
\begin{proof}
We consider the element $\kappa^{[p]}_{n}$ in $\rmH^{1}(\QQ,\rmN_{n})$. We have verified that $\mathrm{res}_{v}(\kappa^{[p]}_{n})$ lies in the finite part for all $v\nmid p\rmN$. On the other hand, the same proofs for \cite[Lemma 2.8]{Flach} and \cite[Lemma 2.10]{Flach} carry over here and shows that $\mathrm{res}_{v}(\kappa^{[p]}_{n})$ lies in $ \rmH^{1}_{f}(\QQ_{v},\rmM_{n})$ for $v\mid \rmN$. Under the local Tate duality
\begin{equation*}
\langle\cdot,\cdot\rangle_{v}: \rmH^{1}(\QQ_{v},\rmN_{n})\times \rmH^{1}(\QQ_{v},\rmM_{n})\rightarrow \calO_{n},
\end{equation*}
it is well--known that the local conditions $\rmH^{1}_{f}(\QQ_{v},\rmN_{n})$ are orthogonal to  $\rmH^{1}_{f}(\QQ_{v},\rmM_{n})$. At the place $p$, it induces a perfect pairing 
\begin{equation*}
\langle\cdot,\cdot\rangle_{p}: \rmH^{1}_{\mathrm{sin}}(\QQ_{p},\rmN_{n})\times \rmH^{1}_{\mathrm{fin}}(\QQ_{p},\rmM_{n})\rightarrow \calO_{n}.
\end{equation*}
Let $s$ be any element in $\rmH^{1}_{f}(\QQ, \rmM_{n})$, then by the global class field theory, we have
\begin{equation*}
\sum_{v}\langle \mathrm{res}_{v}(s), \mathrm{res}_{v}(\kappa^{[p]}_{n})\rangle_{v}=0.
\end{equation*}
This reduces to $\langle\mathrm{res}_{p}(s), \mathrm{res}_{p}(\kappa^{[p]}_{n})\rangle_{p}=0$ by the discussions above. Since by Lemma \ref{index}, we know that $\eta\rmH^{1}_{\sin}(\QQ_{p}, \rmN_{n})$ is contained in the line generated by $\partial_{p}(\kappa^{[p]}_{n})$. Therefore $\eta\res_{p}(s)$ has to vanish. Therefore $\eta\rmH^{1}_{f}(\QQ, \rmM_{n})\subset \rmH^{1}_{\{p\}}(\QQ,\rmM_{n})$ follows.
\end{proof}

\begin{lemma}\label{lm2}
Let $\rmF_{n}=\QQ(\rmM_{n})$ be the splitting field for $\rmM_{n}$ and $\Delta_{n}=\Gal(\rmF_{n}/\QQ)$. Then we have
\begin{equation*}
\rmH^{1}_{\{p\}}(\QQ, \rmM_{n})\subset \rmH^{1}(\Delta_{n}, \rmM_{n})
\end{equation*}
where $\rmH^{1}(\Delta_{n}, \rmM_{n})$ is considered as a subgroup of $\rmH^{1}(\QQ, \rmM_{n})$ via inflation.
\end{lemma}
\begin{proof}
Let $s\in \rmH^{1}_{\{p\}}(\QQ, \rmM_{n})$ and consider the exact sequence 
\begin{equation}
0\rightarrow \rmH^{1}(\Delta_{n}, \rmM_{n})\rightarrow \rmH^{1}(\QQ, \rmM_{n})\rightarrow \rmH^{1}(\rmF_{n}, \rmM_{n})^{\Delta_{n}}=\Hom_{\Delta_{n}}(\rmG_{\rmF_{n}}, \rmM_{n})\rightarrow0.
\end{equation}
Let $\psi: \rmG_{\rmF_{n}}\rightarrow \rmM_{n}$ be the image of $s$ in $\Hom_{\Delta_{n}}(\rmG_{\rmF_{n}}, \rmM_{n})$. Then we need to show that $\psi=0$. Let $\tilde{s}$ be the cocycle representing $s$ and let $\rmF^{\prime}_{n}$ be the fixed field of the kernel of $\tilde{s}$. Let $\Gamma$ be $\Gal(\rmF^{\prime}_{n}/\rmF_{n})$, then it is clear that $\psi$ factors through $\psi: \Gamma\rightarrow \rmM_{n}$. Let $\tau$ be a Frobenius element over $p$ in $\Delta_{n}$ and fix a lift $\tau^{\prime}$ to $\Gal(\rmF^{\prime}_{n}/\QQ)$. Let $g$ be any element in $\Gamma$. By the Chebatorev density theorem we can find a place $v^{\prime}$ of $\rmF^{\prime}_{n}$ such that $\mathrm{Fr}_{\rmF^{\prime}_{n}/\rmF_{n}}(v^{\prime})=\tau^{\prime}g$. Let $v$ be the place under $v^{\prime}$ in $\rmF_{n}$ which necessarily lie over $p$. 

Since $s_{p}:=\mathrm{res}_{p}(s)$ is trivial, $\tilde{s}\vert_{\Gal(\rmF^{\prime}_{n, v^{\prime}}/\QQ_{p})}$ is a coboundary. Thus we have
\begin{equation*}
\tilde{s}(\tau^{\prime}g)\in (\tau^{\prime}g-1)\rmM_{n}=(\tau-1)\rmM_{n}. 
\end{equation*} 
Taking $g=1$ gives $\tilde{s}(\tau^{\prime})\in (\tau-1)\rmM_{n}$. On the other hand, the cocycle relation gives
\begin{equation*}
\tilde{s}(\tau^{\prime}g)=\tilde{s}(\tau^{\prime})+\tau\tilde{s}(g).
\end{equation*}
It follows then $\tau\tilde{s}(g)\in (\tau-1)\rmM_{n}$. Hence $\tilde{s}(g)\in (\tau-1)\rmM_{n}$ for any $g\in\Gamma$ as $(\tau-1)\tilde{s}(g)\in (\tau-1)\rmM_{n}$. Thus the image of $\psi$ lie in $(\tau-1)\rmM_{n}$. Note that $\psi$  is $\Delta_{n}$-equivariant and $\rmM_{n}$ is irreducible. Since $\rmM_{n}\neq (\tau-1)\rmM_{n}$, the element $\psi$ is zero. 
\end{proof}

\begin{myproof}{Theorem}{\ref{main}}
For each $n$, we pick an $n$-admissible prime for $f$. By the previous Lemma \ref{lm1} and Lemma \ref{lm2}, we have
\begin{equation*}
\eta \rmH^{1}_{f}(\QQ, \rmM_{n}) \subset \rmH^{1}(\Delta_{n}, \rmM_{n}).
\end{equation*}
By the second assumption in the statement of the Theorem, we know $\rmH^{1}(\Delta_{n}, \rmM_{n})$ is trivial. Thus $\rmH^{1}_{f}(\QQ, \rmM_{n})$ is indeed annihilated by $\eta$ for each $n$. Hence $\rmH^{1}_{f}(\QQ, \Ad^{0}(\calM_{\rho}))$ has length less than or equal to $\nu=\ord_{\lambda}(\eta)$.
\end{myproof} 

To prove the main theorem, we will study the Bloch--Kato Selmer group $\rmH^{1}_{f}(\QQ, \mathrm{Ad}^{0}(\calM_{\rho}))$ from the perspective of deformation theory of the residual representation $\overline{\rho}$. Although the Bloch--Kato Selmer group itself has less connection with the deformation theory of the residual representation $\overline{\rho}$, we can introduce a smaller Selmer group as follows. We define the local condition $\rmH^{1}_{\mathrm{new}}(\QQ_{v}, \Ad^{0}(\calM_{\rho}))$ and $\rmH^{1}_{\mathrm{new}}(\QQ_{v}, \rmM_{n})$ as in \cite[Definition 3.6]{Lun}, see also \cite[(3.3), (3.4)]{KO}.  Then we define the Selmer group 
\begin{equation*}
\rmH^{1}_{\calS}(\QQ, \Ad^{0}(\calM_{\rho}))  
\end{equation*}
to be 
\begin{equation*}
\mathrm{ker}\{\rmH^{1}(\QQ,\Ad^{0}(\calM_{\rho}))\rightarrow \prod_{v\not\in \Sigma^{-}_{\mathrm{mix}}}\frac{\rmH^{1}(\QQ_{v}, \Ad^{0}(\calM_{\rho}))}{\rmH^{1}_{f}(\QQ_{v},\Ad^{0}(\calM_{\rho}))}\times \prod_{v\in \Sigma^{-}_{\mathrm{mix}}}\frac{\rmH^{1}(\QQ_{v}, \Ad^{0}(\calM_{\rho}))}{\rmH^{1}_{\new}(\QQ_{v},\Ad^{0}(\calM_{\rho}))}\}.
\end{equation*}

On the other hand, consider the global deformation problem given by
\begin{equation*}
\calS_{\mix}:=(\overline{\rho}, \chi_{l}, \Sigma^{+}\cup \Sigma^{-}_{\ram}\cup\Sigma^{-}_{\mix}\cup\{l\}, \{\calD_{v}\}_{v\in \Sigma^{+}\cup \Sigma_{\ram}\cup\Sigma_{\mix}\cup\{l\}})
\end{equation*}
 that classifies the deformations of $\overline{\rho}$ over an $\calO$-algebra which satisfy the following local deformation conditions:
\begin{enumerate}
\item for $v=l$, $\calD_{l}$ classifies deformations which are Fontaine--Laffaille crystalline;
\item for $v\in \Sigma^{-}_{\ram}\cup \Sigma^{+}$, $\calD_{v}$ is the local deformation problem that classifies deformations that are minimally ramified;
\item for $v\in \Sigma^{-}_{\mix}$,  $\calD_{v}=\calD^{\new}_{v}$ is the local deformation problem that classifies deformations that are new in the sense of \cite[Definition 3.6]{Lun}.
\end{enumerate}
It can be seen that the local deformation problem $\calD^{\new}_{v}$ at $v\in \Sigma^{-}_{\mix}$ agrees with that of $\calD^{\ram}_{v}$ in the sense of \cite[Definition 3.51]{LTXZZa}.
This global deformation problem is represented by the deformation ring $\rmR_{\mix}$. Moreover it follows from \cite[Proposition 4.1]{Lun}, see also \cite[Theorem 3.14]{KO} that there is an isomorphism
\begin{equation}
\rmR_{\mix}\cong \TT_{\frakm}.
\end{equation}
Consider the congruence number $\eta_{f}(\rmN^{+},\rmN^{-})$ defined in \cite[\S2.2]{PW} that detects congruences between $f$ and modular forms in $\rmS_{2}(\Gamma_{0}(\rmN))$ which are new at primes dividing $\rmN^{-}$, the $\lambda$-valuation $\ord_{\lambda}(\eta_{f}(\rmN^{+},\rmN^{-}))$ agrees with the length of $\rmH^{1}_{\calS}(\QQ, \Ad^{0}(\calM_{\rho}))$ as a consequence of the above $\rmR=\rmT$ theorem, see \cite[Theorem 4.3]{Lun}, \cite[\S 4.3]{KO}. It also follows that $\calZ_{\rmN^{+}}(\overline{\rmB})=\calO_{\lambda}[\rmZ_{\rmN^{+}}(\overline{\rmB})]_{\frakm}
$ is free over $\TT_{\frakm}$ of rank $1$. Hence \cite[4.17]{DDT} implies that $\eta_{f}(\rmN^{+},\rmN^{-})$ can be chosen to be the Petersson norm $\calP(f^{\dagger})$ of $f^{\dagger}$. Thus we have
\begin{equation*}
\mathrm{length}_{\calO}\phantom{.}\rmH^{1}_{\calS}(\QQ, \Ad^{0}(\calM_{\rho}))=\ord_{\lambda}(\calP(f^{\dagger}))
\end{equation*}
On the other hand, we have an inclusion
\begin{equation*}
\rmH^{1}_{\calS}(\QQ, \Ad^{0}(\calM_{\rho}))\hookrightarrow \rmH^{1}_{f}(\QQ, \Ad^{0}(\calM_{\rho}))
\end{equation*}
by the definitions of the local conditions defining these Selmer groups. Indeed, one can show that
\begin{equation*}
\rmH^{1}_{\new}(\QQ_{v},\Ad^{0}(\calM_{\rho}))
\end{equation*}
is trivial by \cite[Proposition 4.4]{KO}.
Hence we have
\begin{equation*}
\mathrm{length}_{\calO}\phantom{.}\rmH^{1}_{f}(\QQ, \Ad^{0}(\calM_{\rho}))\geq \mathrm{length}_{\calO}\phantom{.}\rmH^{1}_{\calS}(\QQ, \Ad^{0}(\calM_{\rho}))=\ord_{\lambda}(\calP(f^{\dagger})).
\end{equation*}

\begin{theorem}
Under the assumption of Theorem \ref{main}, we have an equality
\begin{equation*}
\mathrm{length}_{\calO}\phantom{.}\rmH^{1}_{f}(\QQ, \mathrm{Ad}^{0}(\calM_{\rho}))=\ord_{\lambda}(\calP(f^{\dagger})).
\end{equation*}
\end{theorem}
\begin{proof}
By the previous Theorem \ref{main}, we have the following inequality
\begin{equation*}
\ord_{\lambda}(\calP(f^{\dagger}))\geq \mathrm{length}\phantom{.}\rmH^{1}_{f}(\QQ, \mathrm{Ad}^{0}(\calM_{\rho})).
\end{equation*}
By the above discussion we have
\begin{equation*}
\mathrm{length}_{\calO}\phantom{.}\rmH^{1}_{f}(\QQ, \Ad^{0}(\calM_{\rho}))\geq \mathrm{length}_{\calO}\phantom{.}\rmH^{1}_{\calS}(\QQ, \Ad^{0}(\calM_{\rho}))=\ord_{\lambda}(\calP(f^{\dagger})).
\end{equation*}
We have finally arrived at the desired equality $\mathrm{length}_{\calO}\phantom{.}\rmH^{1}_{f}(\QQ, \Ad^{0}(\calM_{\rho}))= \ord_{\lambda}(\calP(f^{\dagger}))$.
\end{proof}

\end{document}